\documentclass[oneside, a4paper, 12pt]{article}
\usepackage{fullpage}
\usepackage[colorlinks=true]{hyperref}
\usepackage{amsmath}
\usepackage{amssymb}
\usepackage{amsthm}
\usepackage{graphicx}
\usepackage{caption}
\usepackage{tikz}
\usepackage{amsmath,amssymb,amsfonts}
\usepackage{dsfont}
\usepackage{mathtools}
\usepackage{enumerate}
\usepackage[font=small,labelfont=bf]{caption}
\usepackage{eucal}
\usepackage{mathrsfs}
\usepackage{stmaryrd}
\usepackage{verbatim}
\usepackage{manfnt}
\usepackage{array}
\usetikzlibrary{patterns,cd}
\usepackage{layout} 
\usepackage[all]{xy}
\usepackage{pigpen}

\numberwithin{equation}{section}

\newcommand{\po}{\ar@{}[dr]|{\text{\pigpenfont R}}}
\newcommand{\pb}{\ar@{}[dr]|{\text{\pigpenfont J}}}

\begingroup
\catcode`\&=13
\gdef\pampmatrix{%
  \begingroup
  \let&=\amsamp
  \begin{pmatrix}%
}
\gdef\endpampmatrix{\end{pmatrix}\endgroup}
\endgroup

\DeclareMathOperator{\Sing}{Sing}
\DeclareMathOperator{\Rep}{Rep}
\DeclareMathOperator{\Hom}{Hom}
\DeclareMathOperator{\Fun}{Fun}
\DeclareMathOperator{\HH}{HH}

\DeclareMathOperator{\Sym}{Sym}
\DeclareMathOperator{\Alt}{Alt}

\newcommand{\Kdef}{\mathcal{K}^{\mathrm{def}}}
\newcommand{\Rdef}{\mathcal{R}^{\mathrm{def}}}

\newcommand{\CMon}{\mathrm{CMon}}
\newcommand{\Mod}{\mathcal{M}\mathrm{od}}
\newcommand{\CAlg}{\mathrm{CAlg}}
\newcommand{\cCAlg}{\mathrm{cCAlg}}

\newcommand{\Spc}{\mathrm{Spc}}

\newcommand{\Topp}{\mathrm{Top}}
\newcommand{\Fin}{\mathrm{Fin}}
\newcommand{\sset}{\mathrm{sSet}}

\newcommand{\Kan}{\mathrm{Kan}}

\newcommand{\R}{\mathbb{R}}

\newcommand{\Z}{\mathbb{Z}}
\newcommand{\C}{\mathbb{C}}
\newcommand{\N}{\mathbb{N}}

\newcommand{\Vect}{\mathrm{Vect}}

\newcommand{\colim}{\operatorname{colim}}
\newcommand{\Map}{\operatorname{Map}}
\newcommand{\co}{\colon\thinspace}
\newcommand{\op}{\mathrm{op}}

\theoremstyle{plain}

\newtheorem{theorem}{Theorem}[section]
\newtheorem*{theorem*}{Theorem}
\newtheorem{proposition}[theorem]{Proposition}
\newtheorem{lemma}[theorem]{Lemma}
\newtheorem{corollary}[theorem]{Corollary}

\theoremstyle{definition}
\newtheorem{definition}[theorem]{Definition}
\newtheorem{example}[theorem]{Example}
\newtheorem{construction}[theorem]{Construction}
\newtheorem{variant}[theorem]{Variant}
\newtheorem{notation}[theorem]{Notation}
\newtheorem*{notation*}{Notation}

\theoremstyle{remark}
\newtheorem{remark}[theorem]{Remark}

\usepackage[all]{xy}
\SelectTips{cm}{10}
\CompileMatrices

\DeclareMathAlphabet{\mathpzc}{OT1}{pzc}{m}{it}

\title{Spaces of homomorphisms, formality and Hochschild homology}
\author{Simon Gritschacher \\ {\small \itshape Mathematisches Institut der Ludwig-Maximilians-Universit{\"a}t M{\"u}nchen} \\ {\small \itshape Theresienstrasse 39, 80333, Munich, Germany} \\ {\small {\itshape Email:} \url{simon.gritschacher@math.lmu.de}}}
\date{\today}

\begin{document}

\maketitle

\begin{abstract}
Let $G$ be a discrete group. The topological category of finite dimensional unitary representations of $G$ is symmetric monoidal under direct sum and has an associated $\mathbb{E}_\infty$-space $\Kdef(G)$. We show that if $G$ and $A$ are finitely generated groups and $A$ is abelian, then $\Kdef(G\times A)\simeq \Kdef(G)\otimes \widehat{A}$ as $\mathbb{E}_\infty$-spaces, where $\widehat{A}$ is the Pontryagin dual of $A$. We deduce a homology stability result for the homomorphism varieties $\Hom(G\times \Z^r,U(n))$ using the local-to-global principle for homology stability of Kupers--Miller. For a finitely generated free group $F$ and a field $k$ of characteristic zero, we show that the singular $k$-chains in $\Kdef(F)$ are formal as an $\mathbb{E}_\infty$-$k$-algebra. Using this we describe the equivariant homology of $\Hom(F \times A,U(n))$ for every $n$ in terms of higher Hochschild homology of an explicitly determined commutative $k$-algebra. As an example we show that $\Hom(F\times \Z^r,U(2))$ is $U(2)$-equivariantly formal for every $r$ and we compute the Poincar{\'e} polynomial.
\end{abstract}

\setcounter{tocdepth}{1}


\section{Introduction} \label{sec:intro}

Let $G$ be a finitely generated discrete group and $U(m)$ the Lie group of $m\times m$ unitary matrices. We can consider the representation variety $\Hom(G,U(n))$ and the compact character variety $\Hom(G,U(n))/U(n)$, both of which are interesting moduli spaces from the viewpoint of differential geometry \cite{BFM, CohenStafa, Goldman}. In this paper we use methods from homotopy theory to study the homology of these varieties (in their natural Hausdorff topology) when $G$ is the direct product of two groups one of which is abelian.

The variety $\Hom(G,U(n))$ can depend in a complicated way on the group $G$. For example, if $N\to G\to G/N$ is a group extension, the representation theory of $G$ can be much more complicated than that of $N$ and $G/N$, and there is in general not an easy way of relating $\Hom(N,U(n))$, $\Hom(G/N,U(n))$ and $\Hom(G,U(n))$. This can make it difficult to study the homology of $\Hom(G,U(n))$ in a systematic manner. 

For an abelian group $A$, however, the complex representation theory is so simple that a fully homotopy theoretic description of $\Hom(A,U(n))$ can be achieved. All irreducible representations are $1$-dimensional, and every homomorphism $A\to U(n)$ is determined by an orthogonal decomposition of $\C^n$ and a collection of $1$-dimensional characters of $A$. This information can be organised best by looking at $\Hom(A,U(n))$ not in isolation, but by considering all finite dimensional representations of $A$ at once and the isomorphisms between them.

For any group $G$, the category $\Rep_U(G)$ of finite dimensional unitary representations is naturally a topological category, and taking direct sums of representations equips it with a symmetric monoidal structure. Its topological nerve $\Kdef(G):=\mathrm{N}(\Rep_U(G))$ is an $\mathbb{E}_\infty$-space, whose group-completion is Carlsson's unitary deformation $K$-theory of $G$ \cite{Carlsson, Lawson, RThesis}. For a trivial group $\Kdef(1)$ is the familiar $\mathbb{E}_\infty$-space $\bigsqcup_{m\geq 0} BU(m)$.

Our starting point in this paper is the observation (Proposition \ref{prop:key}) that if $G$ is a group and $A$ is an abelian group, both finitely generated, then twisting representations of $G$ by characters of $A$ induces an equivalence of $\mathbb{E}_\infty$-spaces
\begin{equation} \label{eq:keyequivalence}
\Kdef(G) \otimes \widehat{A} \simeq \Kdef(G \times A)\,,
\end{equation}
where $\widehat{A}=\Hom(A,U(1))$ is the character group of $A$. On the left hand side we use the tensoring of $\mathbb{E}_\infty$-spaces over spaces; equivalently, it can be read as the factorization homology $\int_{\widehat{A}} \Kdef(G)$. After group completion, (\ref{eq:keyequivalence}) may be viewed as a consequence of Lawson's product formula \cite{Lawson}, but the crucial point here is that no passage to group completions is needed. The homology of $\Kdef(G\times A)$ still encodes in a convenient way the equivariant homology of $\Hom(G\times A,U(n))$ for all $n$ at once, as
\[
H_\ast(\Kdef(G\times A);k) \cong \bigoplus_{n\geq 0} H^{U(n)}_\ast(\Hom(G\times A,U(n));k)\,.
\]
In this way, (\ref{eq:keyequivalence}) gives an effective method of studying the homology of $\Hom(G\times A,U(n))$ in relation to the homology of $\Hom(G,U(n))$. In particular, this can be applied to free abelian groups, i.e.,  $(G,A)=(1,\Z^r)$ (see e.g. \cite{AC,AG,KT21,Stafa} for related work), or more generally to direct products of free and free abelian groups, i.e., $(G,A)=(F_s,\Z^r)$. The latter is an interesting class of right-angled Artin groups, which received attention recently in \cite{Florentino2}. This class contains in particular the fundamental groups of the configuration spaces $\mathrm{Conf}_3(\R^2)$ (the pure braid group $P_3\cong F_2\times \Z$) and $\mathrm{Conf}_2((S^1)^2)$ (the group $F_2\times \Z^2$).

We will now summarise our main results.

\subsubsection*{Homology stability}

\begin{definition}
A sequence of maps $X_0\xrightarrow{f_0} X_1\xrightarrow{f_1} X_2\xrightarrow{f_2} \cdots$ satisfies \emph{homology stability (with coefficient group $C$)} if for every fixed $k\geq 0$ the induced maps $(f_n)_\ast\co H_k(X_n;C)\to H_k(X_{n+1};C)$ are isomorphisms when $n$ is sufficiently large (depending on $k$).
\end{definition}

By taking direct sum with a trivial representation we obtain sequences

\begin{align}
\begin{aligned}
\xymatrix{
\Hom_{\mathds{1}}(G,U(1)) \ar[r]^-{\oplus \mathds{1}} & \Hom_{\mathds{1}}(G,U(2)) \ar[r]^-{\oplus \mathds{1}} & \cdots
} 
\end{aligned} \label{eq:hsintro1} \\
\begin{aligned}
\xymatrix{
\Hom_{\mathds{1}}(G,U(1))/U(1) \ar[r]^-{\oplus \mathds{1}} & \Hom_{\mathds{1}}(G,U(2))/U(2) \ar[r]^-{\oplus \mathds{1}} & \cdots  
}
\end{aligned}  \label{eq:hsintro2}
\end{align}
The space $\Hom(G,U(n))$ may or may not be path-connected depending on $G$ and $n$; the subscript $\mathds{1}$ means that we picked the path-component of the trivial homomorphism. In Section \ref{sec:hs} we prove the following homology stability result with coefficients in $\Z$.

\begin{theorem} \label{thm:main1intro}
If the sequence (\ref{eq:hsintro1}), or (\ref{eq:hsintro2}), satisfies homology stability for a finitely generated group $G$, then it satisfies homology stability for $G\times \Z^r$ for every $r\geq 0$.
\end{theorem}

Since $\Hom(\Z^r,U(n))$ is path-connected for all $r,n$, choosing $G$ to be the trivial group the following corollary is obtained, generalising results of \cite{KT21,RS20}.

\begin{corollary}
For every $r\geq 0$, the sequence of spaces $\Hom(\Z^r,U(n))$, $n\geq 0$ satisfies homology stability with integer coefficients.
\end{corollary}

With rational coefficients homology stability of $\Hom(\Z^r,U(n))$ was first proved by Ramras and Stafa \cite{RS20} using methods from representation homology stability. A different proof, based on finding minimal generating sets for cohomology, was given by Kishimoto and Takeda \cite{KT21}. While the latter give optimal stability bounds in the rational setting, both proofs rely fundamentally on a well known, explicit description of the rational cohomology ring of $\Hom(\Z^r,U(n))$; so these proofs cannot be applied to prove stability with integer coefficients. Our proof is very different from theirs; it is an easy consequence of (\ref{eq:keyequivalence}) and a general homology stability principle for factorisation homology due to Kupers and Miller \cite{KM18}.

The compact character varieties $\Rep(\Z^k,U(n))$, on the other hand, are homeomorphic to the symmetric powers $\mathrm{SP}^n((S^1)^k)$ of a $k$-torus; they are well-known to satisfy homology stability with coefficients in $\Z$ (see e.g. \cite[Example 1.4]{KM18})

\subsubsection*{Spectral sequences and Hochschild homology}

Many computational results on representation varieties are over fields of characteristic zero, justified also by the interest in their algebro-geometric features (as in \cite{FLS}, for example), but their torsion is known to have a rich and complicated structure (see \cite{KT24} for the case of free abelian group representations). A formal consequence of (\ref{eq:keyequivalence}) are certain spectral sequences which provide a computational tool over more general rings. In Section \ref{sec:ss} we prove:

\begin{theorem} \label{thm:ssintro}
Let $k$ be a commutative ring. If $H_\ast(\Kdef(G);k)$, or $H_\ast(\Rdef(G);k)$, is $k$-flat, then there are, respectively, first quadrant spectral sequences for every $n\geq 0$
\begin{alignat}{3}
E^2 & =\HH^{\widehat{A}}_\ast(H_\ast(\Kdef(G);k))_n  && \Longrightarrow H_\ast^{U(n)}(\Hom(G\times A,U(n));k) \label{eq:ssintro1}\\ 
E^2 & =\HH^{\widehat{A}}_\ast(H_\ast(\Rdef(G);k))_n  && \Longrightarrow H_\ast(\Hom(G\times A,U(n))/U(n);k) \label{eq:ssintro2}
\end{alignat}
\end{theorem}
In the theorem, $\Rdef(G)=\bigsqcup_{n\geq 0} \Hom(G,U(n))/U(n)$ is the topological abelian monoid of conjugacy classes of representations, and $H_\ast(\Kdef(G);k)$ and $H_\ast(\Rdef(G);k)$ are viewed as graded commutative $k$-algebras. The $E^2$-page is the higher Hochschild homology of these algebras, in the sense of Loday and Pirashvili \cite{Pirashvili}. We will apply this spectral sequence to compute the Poincar{\'e} series of $\Hom(F_2\times \Z^r,U(2))/U(2)$ for any $r$ at the end of Section \ref{sec:example}. An application of (\ref{eq:ssintro1}) to the integral cohomology of spaces of commuting elements in $U(n)$ will be given in a forthcoming paper.

\subsubsection*{Free groups and formality}

In the second part of the paper we will focus on the case where $G=F_s$ is a finitely generated free group. A key observation will be the following formality result which we prove in Section \ref{sec:formality}. 

\begin{proposition} \label{prop:formalityintro}
Let $k$ be a field of characteristic zero. Then the singular $k$-chains in $\Kdef(F_s)$ are formal as an $\mathbb{E}_\infty$-$k$-algebra.
\end{proposition}

The $\mathbb{E}_\infty$-space $\Kdef(F_s)$ is just the mapping space $\Map(\vee^s S^1,\bigsqcup_{m\geq 0} BU(m))$, but formality is not so obvious. We have no reason to expect that $\Kdef(G)$ is formal in this sense for a great many other groups $G$.  As a corollary of Proposition \ref{prop:formalityintro} we obtain an entirely algebraic description of the equivariant homology of $\Hom(F_s\times \Z^r,U(n))$ for all $r,s,n\geq 0$.

\begin{theorem} \label{thm:hhintro}
Let $k$ be a field of characteristic zero. Let $A$ be a finitely generated abelian group, and let $F_s$ be a free group of finite rank $s$. Then there is an isomorphism of graded commutative $k$-algebras
\[
H_\ast(\Kdef(F_s\times A);k)\cong \HH^{\widehat{A}}_\ast(H_\ast(\Kdef(F_s);k))\, .
\]
\end{theorem}
A description of the algebra $H_\ast(\Kdef(F_s);k)$ as a quotient of a symmetric algebra is given in Theorem \ref{thm:quotientoffreealgebra}. Through a suitable grading on both the homology of $\Kdef(F_s\times A)$ and the Hochschild homology one can extract from this theorem the equivariant homology of $\Hom(F_s\times A,U(n))$ for each individual $n$. For $n=2$ and $A=\Z^r$ the computations will be carried out in great detail in Section \ref{sec:example}. The main result is the computation of the Poincar{\'e} series of $\Hom(F_s\times \Z^r,U(2))$ in Theorem \ref{thm:u2}.

We should point out that we have no analogue of Proposition \ref{prop:formalityintro} or Theorem \ref{thm:hhintro} for character varieties, i.e., for $\Rdef(F_s)$ instead of $\Kdef(F_s)$. The homology of the free group character varieties $\Hom(F_s,U(n))/U(n)$ is largely unknown.

\subsubsection*{Organisation of the paper}

In Section \ref{sec:gammaspace} we recall Segal's $\Gamma$-spaces and how they model $\mathbb{E}_\infty$-spaces. We explain how to pass from the $1$-categorical constructions which occur naturally, to the more flexible $\infty$-categorical descriptions used later in the paper. In Section \ref{sec:defktheory} we define the deformation $K$-theory space based on $\Gamma$-spaces and make the key observation (\ref{eq:keyequivalence}). In Section \ref{sec:hs} we apply the local-to-global principle of Kupers-Miller to prove our homology stability result of Theorem \ref{thm:main1intro}. Examples are then discussed in Section \ref{sec:examples}. In Section \ref{sec:ss} we define Hochschild homology and discuss the spectral sequences of Theorem \ref{thm:ssintro}. In Section \ref{sec:formality} we state Theorem \ref{thm:hhintro}, the proof of which occupies Section \ref{sec:proofofformality}. Finally, in Section \ref{sec:example} we discuss applications of Theorem \ref{thm:hhintro}.

\section*{Acknowledgements}

I am indebted to Markus Hausmann for sharing his ideas, especially at the early stages of the project. I also thank Markus Land for many helpful discussions, and Daniel Ramras for comments on an earlier version and for sharing his unpublished notes with me.


\section{Recollection of $\Gamma$-spaces and $\mathbb{E}_\infty$-spaces} \label{sec:gammaspace}

\subsection{$\Gamma$-spaces}

For an integer $k\geq 0$ let $\langle k\rangle$ denote the finite based set $\{\ast,1,\dots,k\}$ with basepoint $\ast$. Let $\Fin_\ast$ denote the category whose objects are the sets $\langle k\rangle$ for $k\geq 0$ and whose morphisms are based maps $\langle k\rangle \to \langle l\rangle$. Let $\Topp$ denote the category of compactly generated spaces, as defined for example in \cite[Appendix A]{Global}.

\begin{definition} \label{def:gammaspace}
A \emph{$\Gamma$-space} is a functor $A \co \Fin_\ast \to \Topp$ such that $A(\langle 0\rangle)$ is a point.
\end{definition}

A $\Gamma$-space can be prolonged to a continuous functor on the category of based spaces $\Topp_\ast$, and we denote the prolonged functor by $A$ as well. Concretely, $A$ can be defined on a pointed space $X\in \Topp_\ast$ by the coend $\int^{\Fin_\ast} A\times \textnormal{Map}_\ast(-,X)$, that is, by the quotient of
\[
\bigsqcup_{\langle k\rangle \in \Fin_\ast} A(\langle k\rangle)\times  \textnormal{Map}_\ast(\langle k\rangle,X)
\]
by the equivalence relation generated by
\begin{equation} \label{eq:equivalencerelation}
(A(\alpha)(x),f)\sim (x,f \alpha)
\end{equation}
for all $x\in A(\langle k\rangle)$, $f\in \textnormal{Map}_\ast(\langle l\rangle,X)$ and $\alpha\co \langle k\rangle \to \langle l\rangle$ in $\Fin_\ast$.

We may also evaluate $A$ on a pointed simplicial set $X\in \sset_\ast$. We do this by evaluating $A$ levelwise on $X$, viewing based sets as based discrete topological spaces. This yields a simplicial space $A(X)$. For geometric realisation there is then a homeomorphism
\begin{equation} \label{eq:cgammaspaceofrealisation}
A(|X|)\cong |A(X)|\, ,
\end{equation}
see, for example, \cite[Proposition B.29]{Global}.

For $\langle k\rangle \in \Fin_\ast$ and $i=1,\dots,k$, let $\rho^i\co \langle k\rangle \to \langle 1\rangle$ be the map defined by $\rho^i(j)=1$ if $i=j$ and $\rho^i(j)=\ast$ otherwise.

\begin{definition} \label{def:special}
A $\Gamma$-space $A$ is called \emph{special}, if for all $\langle k\rangle \in \Fin_\ast$ the maps $\rho^1,\dots,\rho^k$ induce a weak equivalence $A(\langle k\rangle)\xrightarrow{\sim} A(\langle 1\rangle)^{\times k}$.
\end{definition}

Let $\nabla\co \langle 2\rangle\cong \langle 1\rangle \vee \langle 1\rangle \to \langle 1\rangle$ be the fold map. If $A$ is special, the zig-zag
\[
\xymatrix{
A(\langle 1\rangle)\times A(\langle 1\rangle) \,& & \ar[ll]_-{(A(\rho^1),A(\rho^2))}^-{\sim} A(\langle 2\rangle) \ar[r]^-{A(\nabla)} & A(\langle 1\rangle)
}
\]
exhibits $A(\langle 1\rangle)$ as a commutative monoid in the homotopy category. Special $\Gamma$-spaces model $\mathbb{E}_\infty$-spaces and we think of $A(\langle 1\rangle)$ as the underlying space of the $\mathbb{E}_\infty$-space.

For $\langle k\rangle\in \Fin_\ast$ we denote by $\langle k\rangle^\circ$ the set $\langle k\rangle$ minus the basepoint. For a set $S$ we denote by $S_+$ the set $S$ with an extra basepoint added.

\begin{definition}
A $\Gamma$-space is called \emph{cofibrant} if for every $\langle n\rangle \in \Fin_\ast$ the latching map
\[
\textnormal{colim}_{S\subsetneq \langle n\rangle^\circ} A(S_+) \to A(\langle n\rangle)
\]
is a $\Sigma_n$-cofibration for the evident action of the symmetric group $\Sigma_n$ on the source and target, see \cite[Definition B.33]{Global}.
\end{definition}

The $\Gamma$-spaces considered in this article will have this property.

\subsection{Tensoring a $\Gamma$-space with a space} \label{sec:tensoring}

We now recall how a $\Gamma$-space can be tensored with a space. We will do this first for $\Gamma$-objects in simplicial sets, that is, functors $A\co \Fin_\ast\to \sset$ with $A(\langle 0\rangle)=\ast$. We will return to $\Gamma$-spaces at the end of this subsection.

A functor $A\co \Fin_\ast\to \sset$ can be prolonged to the category of pointed simplicial sets via homotopy left Kan extension along the inclusion $\Fin_\ast\subseteq \sset_\ast$. More precisely, we evaluate $A$ on any based set, not just finite ones, by a filtered colimit. For $X\in \sset_\ast$ we may then define $A(X)$ by evaluating $A$ levelwise on $X$ and taking the diagonal of the resulting bisimplicial set.

\begin{construction} \label{cons:tensoring}
For $X\in \sset_\ast$ and a $\Gamma$-object $A\co \Fin_\ast\to \sset$ we define a new $\Gamma$-object $A_X\co \Fin_\ast\to \sset$ by
\[
A_X(\langle k\rangle)=A(\langle k\rangle \wedge X)\,.
\]
\end{construction}

If $A$ is special (in the sense of Definition \ref{def:special}), then so is $A_X$ for any based simplicial set. One can show this first for finite $X$, and then for general $X$ using the fact that filtered colimits commute with finite products in simplicial sets.

Let $A\co \Fin_\ast\to \sset$ be a special $\Gamma$-object, $X$ a simplicial set and $X_+$ the based simplicial set obtained by adjunction of a disjoint basepoint. We will now explain, in a somewhat ad hoc manner, that the assignment $(A,X)\to A_{X_+}$ of Construction \ref{cons:tensoring} enjoys the universal property of a tensor in the category of $\mathbb{E}_\infty$-spaces. To make this precise it will be convenient to use $\infty$-categorical language.

Our reference for $\infty$-categories will be \cite{HTT}; we let $\Spc$ and $\Spc_\ast$ denote the $\infty$-categories of spaces and based spaces, respectively, and $\mathrm{N}$ denotes the homotopy coherent nerve.

\begin{lemma}
The functor
\[
\mathrm{Fun}(\Fin_\ast,\sset) \times \sset_\ast\to \mathrm{Fun}(\Fin_\ast,\sset),\quad (A,X)\mapsto A_X
\]
refines to a functor of $\infty$-categories
\[
T\co \Fun(\mathrm{N}\Fin_\ast,\Spc) \times \Spc_\ast \to \Fun(\mathrm{N}\Fin_\ast,\Spc)\, .
\]
\end{lemma}
\begin{proof}
We view $\sset$ as a simplicial model category equipped with the Quillen model structure. The functor category $\Fun(\Fin_\ast,\sset)$ is equipped with the projective model structure; this model structure is again simplicial, see \cite[Section A.3.3]{HTT}. The adjoint of $(A,X)\mapsto A_X$ is the functor
\[
\Fun(\Fin_\ast,\sset)\times \sset_\ast \times \Fin_\ast  \to \sset
\]
which is given by prolongation $\Fun(\Fin_\ast,\sset)\to \Fun(\sset_\ast,\sset)$, the smash product $\wedge\co \sset_\ast \times \Fin_\ast\to \sset_\ast$ and evaluation $\Fun(\sset_\ast,\sset)\times \sset_\ast\to \sset$, all of which are canonically $\sset$-enriched. Hence, their composite is $\sset$-enriched, and by adjunction this shows that $(A,X)\mapsto A_X$ extends canonically to a $\sset$-enriched functor.

Next recall from \cite[4.2.4.4]{HTT} that the $\infty$-category $\Fun(\mathrm{N}\Fin_\ast,\Spc)$ is modelled by the underlying $\infty$-category of the simplicial model category $ \Fun(\Fin_\ast,\sset)$. That means there is a canonical equivalence of $\infty$-categories
\begin{equation} \label{eq:functorcategory}
\mathrm{N} \Fun(\Fin_\ast,\sset)^{\circ}\simeq \Fun(\mathrm{N}\Fin_\ast,\Spc)
\end{equation}
where $\Fun(\Fin_\ast,\sset)^{\circ}\subseteq \Fun(\Fin_\ast,\sset)$ denotes the full subcategory on the bifibrant objects. The sought functor $T$ is now obtained by restricting the $\sset$-enriched functor $(A,X)\mapsto A_X$ to bifibrant objects, composing with a simplicial bifibrant replacement and taking homotopy coherent nerves.
\end{proof}

\begin{notation}
A $\Gamma$-object $A\co \Fin_\ast\to \sset$ gives rise to a functor of $\infty$-categories $\mathrm{N} \Fin_\ast \to \Spc$ by bifibrantly replacing $A$ and applying the nerve. We will denote this functor simply by $\mathrm{N}(A)$.
\end{notation}

With this notation, if $A$ is a $\Gamma$-object and $X\in \Spc_\ast$, then
\[
\mathrm{N}(A_X) \simeq T(\mathrm{N}(A),X)
\]
in the $\infty$-category $\Fun(\mathrm{N}\Fin_\ast,\Spc)$.

Following \cite[Definition 2.4.2.1]{HA} we define the $\infty$-category of $\mathbb{E}_\infty$-spaces as a full subcategory of $\Fun(\mathrm{N} \Fin_\ast,\Spc)$.

\begin{definition} \label{def:cmon}
An \emph{$\mathbb{E}_\infty$-space} is a functor of $\infty$-categories $A\co \mathrm{N} \Fin_\ast\to \Spc$ such that for every $\langle k\rangle \in \Fin_\ast$ the maps $\rho^1,\dots,\rho^k$ (see Definition \ref{def:special}) induce an equivalence $A(\langle k\rangle)\xrightarrow{\sim} A(\langle 1\rangle)^{\times k}$ in the $\infty$-category $\Spc$. We let $\CMon(\Spc)$ denote the $\infty$-category of $\mathbb{E}_\infty$-spaces.
\end{definition}

Because of the equivalence (\ref{eq:functorcategory}) every $\mathbb{E}_\infty$-space is equivalent to one of the form $\mathrm{N}(A)$ for a special $\Gamma$-object $A\co \Fin_\ast\to \mathrm{Kan}$ (where $\mathrm{Kan}$ is the category of Kan complexes).

The $\infty$-category $\CMon(\Spc)$ is presentable, see \cite[3.2.3.5]{HA}; in particular, it has small colimits. It follows that $\CMon(\Spc)$ is canonically tensored over $\Spc$ (see \cite[5.5.1.7]{HTT}) and we let
\[
\otimes \co \CMon(\Spc)\times \Spc\to \CMon(\Spc)
\]
denote the tensoring. Informally, this means that for an $\mathbb{E}_\infty$-space $A$ and $X\in \Spc$ the $\mathbb{E}_\infty$-space $A\otimes X$ is uniquely determined by demanding that $A\otimes (-)$ preserves colimits and that $A\otimes \ast\simeq A$ where $\ast\in \Spc$ is the one-point space.

Let $(-)_+\co \Spc\to \Spc_\ast$ be the left-adjoint of the forgetful functor.

\begin{lemma} \label{lem:tensoring}
Let $T'$ denote the composite functor of $\infty$-categories
\[
\CMon(\mathcal{S}) \times \Spc \xrightarrow{\textnormal{incl}\times (-)_+}  \Fun(\mathrm{N}\Fin_\ast,\Spc)\times \Spc_\ast \xrightarrow{T} \Fun(\mathrm{N}\Fin_\ast,\Spc)\, .
\]
\begin{enumerate}
\item[(i)] The essential image of $T'$ is $\CMon(\Spc)$.
\item[(ii)] As a functor into $\CMon(\Spc)$, $T'$ is canonically equivalent to $\otimes$.
\end{enumerate}
\end{lemma}
\begin{proof}
Item (i) is a consequence of the fact that if $A\co \Fin_\ast\to \Kan$ is special, then so is $A_X$ for any $X\in \Spc_\ast$ and thus $\mathrm{N}(A_X)\in \CMon(\Spc)$.

To prove (ii) it is enough to show, by the universal property of $\Spc$, that $T'$
\begin{itemize}
\item is equivalent to the identity when restricted to the point $\ast\in \Spc$
\item preserves colimits in the second argument.
\end{itemize}

The first item is clear, because $T(-,S^0)\simeq id$. To prove the second item it is enough to check this separately for finite coproducts and for sifted colimits.

Let $X,Y\in \Spc$ and let $A\co \Fin_\ast\to \Kan$ be special. By functoriality, the inclusions $X_+\hookrightarrow X_+\vee Y_+ \hookleftarrow Y_+$ induce maps of $\mathbb{E}_\infty$-spaces
\[
\mathrm{N} (A_{X_+}) \xrightarrow{j_X} \mathrm{N} (A_{X_+\vee Y_+}) \xleftarrow{j_Y} \mathrm{N}(A_{Y_+})
\]
and hence a map from the coproduct in $\CMon(\Spc)$ (also denoted by $\otimes$)
\[
\{j_X,j_Y\}\co \mathrm{N} (A_{X_+}) \otimes \mathrm{N}(A_{Y_+}) \to  \mathrm{N} (A_{X_+\vee Y_+})\, .
\]
We claim that this map is an equivalence, thus exhibiting $\mathrm{N} (A_{X_+\vee Y_+})$ as a coproduct of $\mathrm{N} (A_{X_+})$ and $\mathrm{N} (A_{Y_+})$. This can be checked on underlying spaces, that is, after evaluation at $\langle 1\rangle\in \mathrm{N}\Fin_\ast$. There we have
\[
(\mathrm{N}(A_{X_+}) \otimes \mathrm{N}(A_{Y_+}))(\langle 1\rangle) \xrightarrow{\{j_X,j_Y\}}  \mathrm{N} (A_{X_+\vee Y_+})(\langle 1\rangle) \xrightarrow{\sim} \mathrm{N}(A_{X_+})(\langle 1\rangle)\times \mathrm{N}(A_{Y_+})(\langle 1\rangle)\, .
\]
The second map, which is induced by the evident maps $X_+\leftarrow X_+\vee Y_+\to Y_+$, is an equivalence, because $A$ is special. The composite map is an equivalence, witnessing the fact that the forgetful functor $\CMon(\Spc)\to \Spc$ is symmetric monoidal for the coproduct on $\CMon(\Spc)$ and the Cartesian product on $\Spc$, see \cite[3.2.4.7]{HA}. Hence, the first map is an equivalence as well. This proves preservation of finite coproducts in $\Spc$.

To check preservation of sifted colimits in $\Spc$, we use the fact that sifted colimits in $\CMon(\Spc)$ are reflected by the forgetful functor $\CMon(\Spc)\to \Spc$, see \cite[3.2.3.2]{HA}. This reduces the question to whether $T$ preserves sifted colimits in $\Spc_\ast$.

Through the correspondence of homotopy colimits in a simplicial model category with colimits in the underlying $\infty$-category \cite[4.2.4.1]{HTT}, this eventually reduces the question to whether, for any fixed $\Gamma$-object $A\co \Fin_\ast\to \sset$, the functor $\sset_\ast \to \sset$, $X\mapsto A(X)$ preserves sifted homotopy colimits. But this is so by the very definition of the prolongation of $A$.
\end{proof}

\begin{corollary} \label{cor:tensor}
For every $X\in \Spc$ and special $\Gamma$-object $A\co \Fin_\ast\to \sset$, there is an equivalence of $\mathbb{E}_\infty$-spaces
\[
\mathrm{N}(A_{X_+}) \simeq \mathrm{N} (A)\otimes X
\]
which is natural in $A$ and $X$.
\end{corollary}

The $\Gamma$-spaces considered in Section \ref{sec:defktheory} are naturally defined in $\Topp$; we now explain how to translate them into a simplicial setting. 

First, we have the following analogue of Construction \ref{cons:tensoring}.

\begin{construction} \label{cons:tensoring2}
For $X\in \Topp_\ast$ and a $\Gamma$-space $A\co \Fin_\ast\to \Topp$ we define a new $\Gamma$-space $A_X\co \Fin_\ast\to \Topp$ by setting
\[
A_X(\langle k\rangle)= A(\langle k \rangle \wedge X)\,.
\]
\end{construction}

If $X$ is a finite CW-complex and $A$ is special and cofibrant, then $A_X$ is again special and cofibrant by \cite[Proposition B.54]{Global}.

Let $\textnormal{Sing}\co \Topp\to\sset$ be the singular complex.  If $A$ is a (special) $\Gamma$-space, then
\[
\Sing A \co \Fin_\ast\to \sset
\]
is a (special) $\Gamma$-object in simplicial sets. 

\begin{lemma} \label{lem:singular}
If $A$ is a cofibrant $\Gamma$-space and $X$ is a finite based simplicial set, then there is a natural weak equivalence
\[
(\Sing A)_X \simeq \Sing A_{|X|}\, ,
\]
which is also natural in $A$ and $X$.
\end{lemma}
\begin{proof}
Recall that the geometric realisation of the diagonal of a bisimplicial set is homeomorphic to the successive realisations in the two simplicial directions. Thus, for every $\langle k\rangle\in \Fin_\ast$ there is a homeomorphism
\[
|(\Sing A)_X(\langle k\rangle)|\cong |[n]\mapsto |(\Sing A)(\langle k\rangle\wedge X_n)||\, .
\]
The counit of the adjunction $|-|\dashv \Sing$ gives a levelwise weak equivalence of simplicial spaces
\[
\{ [n]\mapsto |(\Sing A)(\langle k\rangle\wedge X_n)|\} \xrightarrow{\sim} \{[n] \mapsto A(\langle k\rangle\wedge X_n)\}\, .
\]
By \cite[Proposition B.37]{Global}, since $A$ is cofibrant and $X$ is finite, the simplicial space on the right is Reedy cofibrant. Hence, according to \cite[Proposition A.44]{Global} the map above induces a weak equivalence after geometric realisation. By (\ref{eq:cgammaspaceofrealisation}), the geometric realisation of the right hand side is identified with $A_{|X|}(\langle k\rangle)$. So we obtain a weak equivalence $|(\Sing A)_X(\langle k\rangle)|\simeq |(\Sing A_{|X|})(\langle k\rangle)|$. Since all identifications made are natural in $\langle k\rangle$, $A$ and $X$, the lemma follows.
\end{proof}

\subsection{$\mathbb{E}_\infty$-algebras in $\N$-graded spaces} \label{sec:graded}

The $\mathbb{E}_\infty$-spaces of Section \ref{sec:defktheory} are built from spaces of representations and come with a natural $\N$-grading by the degree of a representation. As this grading will be important later on, we will now explain what kind of extra structure we have to put on a $\Gamma$-space in order to describe an $\mathbb{E}_\infty$-algebra in $\N$-graded spaces.

Let $\N$ denote the natural numbers viewed as a discrete $1$-category.

\begin{definition} \label{def:gradedobjects}
Let $\mathcal{C}$ be an $\infty$-category. The $\infty$-category of \emph{$\N$-graded objects in $\mathcal{C}$} is defined by $\mathcal{C}^\N:=\Fun(\mathrm{N}(\N),\mathcal{C})$.
\end{definition}

Objects of $\mathcal{C}^\N$ are sequences $(X_0,X_1,X_2,\dots)$ of objects of $\mathcal{C}$.

Now assume that $\mathcal{C}$ is the underlying $\infty$-category of a symmetric monoidal $\infty$-category $\mathcal{C}^\otimes$.

\begin{notation}
We denote by $\CAlg(\mathcal{C})$ the $\infty$-category of $\mathbb{E}_\infty$-algebras in $\mathcal{C}^\otimes$.
\end{notation}

We want to consider $\mathbb{E}_\infty$-algebras in the category of $\N$-graded objects $\mathcal{C}^\N$. Assume that $\mathcal{C}$ is presentable and that the underlying monoidal product functor $\otimes\co \mathcal{C}\times \mathcal{C}\to \mathcal{C}$ preserves small colimits separately in each variable. Then $\mathcal{C}^{\N}$ can be given a symmetric monoidal structure by \emph{Day convolution}, see \cite[Section 2.2.6]{HA}. On objects Day convolution is informally described by
\[
(X_n)_{n\geq 0} \otimes (Y_n)_{n\geq 0}=\left(\coprod_{k+l=n} X_k\otimes Y_l\right)_{n\geq 0}\,.
\]
The $\infty$-category $\mathcal{C}^\N$ is again presentable \cite[5.5.3.6]{HTT}, and the tensor product preserves small colimits separately in each variable. It thus follows from \cite[3.2.3.5]{HA} that $\CAlg(\mathcal{C})$ as well as $\CAlg(\mathcal{C}^\N)$ are presentable. In particular, there is a left-adjoint, forgetful functor
\[
p_{!}\co \CAlg(\mathcal{C}^\N) \to \CAlg(\mathcal{C})
\]
which sends a graded object $(X_n)_{n\geq 0}$ to its underlying ungraded object $\coprod_{n\geq 0} X_n$.

\begin{definition} \label{def:ntensor}
Let $\N^\otimes$ denote the $1$-category whose objects are tuples
\[
(x_1,\dots,x_m) \in \N^m\,, \quad m\geq 0
\]
and whose morphisms
\[
\alpha\co (x_1,\dots,x_m)\to (y_1,\dots,y_n)
\]
are morphisms $\alpha\co \langle m\rangle\to \langle n\rangle$ in $\Fin_\ast$ satisfying $y_j=\sum_{i\in \alpha^{-1}(j)} x_i$ for all $j=1,\dots,n$.
\end{definition}

There is an obvious forgetful functor $\pi\co \N^\otimes \to \Fin_\ast$. The functor of $\infty$-categories $\mathrm{N}(\N^\otimes)\to \mathrm{N}\Fin_\ast$ induced by $\pi$ describes then the natural numbers under addition as a symmetric monoidal $\infty$-category, see \cite[2.1.2.21]{HA}.

Generally, the $\infty$-category $\CAlg(\mathcal{C}^\N)$ is equivalent to the $\infty$-category of lax symmetric monoidal functors from $\mathrm{N}(\N^\otimes)$ to $\mathcal{C}^\otimes$. When $\mathcal{C}$ is $\Spc$ with Cartesian product, there is the following simpler description: By  \cite[2.4.1.7]{HA} there is an equivalence of $\infty$-categories

\[
\CAlg(\Spc^\N)\simeq \mathrm{Fun}^{\textnormal{lax}}(\mathrm{N}(\N^\otimes),\Spc)\,,
\]
where the right hand side denotes the the $\infty$-category of \emph{lax Cartesian structures} on $\mathrm{N}(\N^\otimes)$. These are functors of $\infty$-categories $A\co \mathrm{N}(\N^\otimes)\to \Spc$ that satisfy the condition that for every $(x_1,\dots,x_m)\in \N^\otimes$ the maps $\rho^i\co \langle m\rangle\to \langle 1\rangle$, $i=1,\dots,m$ (see Definition \ref{def:cmon}) induce an equivalence in $\Spc$
\begin{equation} \label{eq:segal}
A(x_1,\dots,x_m) \xrightarrow{\sim} A(x_1)\times \cdots \times A(x_m)\, .
\end{equation}
This motivates the following definition.

\begin{definition} \label{def:nspace}
An \emph{$\N^\otimes$-space} is a functor $A\co \N^\otimes\to \Topp$ such that $A(\ast)=\ast$ (where $\ast \in \N^0$ is the empty tuple).
\end{definition}

Any $\N^\otimes$-space $A$ has an underlying $\Gamma$-space $\pi_{!}(A)\co \Fin_\ast\to \Topp$ obtained by left Kan extension along the forgetful functor $\pi\co \N^\otimes \to \Fin_\ast$. Using the pointwise formula for left Kan extensions one sees that
\begin{equation} \label{eq:nspacedecomposition}
\pi_{!}(A)(\langle m\rangle) \cong \bigsqcup_{(x_1,\dots,x_m)\in \N^m} A(x_1,\dots,x_m)
\end{equation}
Moreover, it follows that condition (\ref{eq:segal}) is satisfied by a $\N^\otimes$-space $A$ if and only if the underlying $\Gamma$-space $\pi_{!}(A)$ is special.

In summary, we can describe an $\mathbb{E}_\infty$-algebra in $\N$-graded spaces by a special $\Gamma$-space $\Fin_\ast\to \Topp$ (or a special $\Gamma$-object in $\sset$) with the added structure of a decomposition (\ref{eq:nspacedecomposition}) for every $\langle m\rangle\in \Fin_\ast$. The decompositon must be such that a map $\alpha\co \langle m\rangle\to \langle n\rangle$ takes the disjoint summand indexed by $(x_1,\dots,x_m)\in \N^m$ to the disjoint summand indexed by $(y_1,\dots,y_n)\in \N^n$, where $y_j=\sum_{i\in \alpha^{-1}(j)} x_i$. Later on we will not distinguish notationally between an $\N^\otimes$-space $A$ and its underlying $\Gamma$-space.

An $\N^\otimes$-space $A$ can be tensored with a based space of the form $X_+$; for the projection $X_+\to \langle 1 \rangle$ collapsing $X$ to $1\in \langle 1 \rangle$ induces a map
\[
q\co A_{X_+}(\langle m\rangle) \to A_{\langle 1 \rangle}(\langle m\rangle)=A(\langle m \rangle)=\bigsqcup_{(x_1,\dots,x_m)\in \N^m} A(x_1,\dots,x_m)
\]
and hence a decomposition
\[
A_{X_+}(\langle m\rangle)=\bigsqcup_{(x_1,\dots,x_m)\in \N^m} q^{-1}(A(x_1,\dots,x_m))\, .
\]

Note that tensoring the $\N^\otimes$-space $A$ with a based space $X$ does not in general give a graded object, unless the basepoint of $X$ is disjoint.

We may also tensor a $\N^\otimes$-object in $\sset$ with a based simplicial set $X_+$, by first tensoring with the levels of $X_+$ and then taking the diagonal of the resulting bisimplicial set. If the underlying $\Gamma$-space of $A$ is cofibrant, then $A_{|X|_+}$ satisfies the analogue of Lemma \ref{lem:singular}, with the same proof.

Now let $A\co \N^\otimes\to \sset$ be a $\N^\otimes$-object, whose underlying $\Gamma$-object is special. Let $\mathrm{N}(A)\in \CAlg(\Spc^\N)$ be the $\mathbb{E}_\infty$-algebra in $\N$-graded spaces corresponding to it. The same proof as that of Lemma \ref{lem:tensoring}, with $\Fin_\ast$ replaced by $\N^\otimes$ shows

\begin{corollary} \label{cor:tensor2}
For every space $X\in \Spc$, there is an equivalence in $\CAlg(\Spc^\N)$
\[
\mathrm{N}(A_{X_+})\simeq \mathrm{N}(A)\otimes X
\]
which is natural in $A$ and $X$.
\end{corollary}


\section{The $\Gamma$-space for deformation $K$-theory} \label{sec:defktheory}

Let $G$ be a discrete group. The mapping space $\Map(G,U(n))$ is given the compact open topology, equivalently the product topology by writing it as $U(n)^G$. We give
\[
\Hom(G,U(n))\subseteq \Map(G,U(n))
\]
the subspace topology. It makes $\Hom(G,U(n))$ into a compact Hausdorff space. The group $U(n)$ operates continuously on $\Hom(G,U(n))$ by conjugation.

We now recall the $\Gamma$-space for unitary deformation $K$-theory of $G$ due to Lawson \cite{Lawson}; it packages the homomorphism spaces $\Hom(G,U(n))$ for various $n$ into a single homotopical object. Let $\C^n$, $0\leq n\leq \infty$ be equipped with the standard Hermitian inner product. For $m< \infty$ let $V(n,\C^m)$ denote the compact Stiefel manifold of $n$-dimensional orthonormal frames in $\C^m$. It carries a free action of $U(n)$ whose quotient space is the Grassmannian $Gr(n,\C^m)$. If $m=\infty$, then $V(n,\C^\infty)$ shall be viewed as the colimit in $\Topp$
\[
V(n,\C^\infty)=\textnormal{colim}_{m} V(n,\C^m)\, .
\]
Consider now the fibre bundle
\[
\Hom(G,U(n)) \to V(n,\C^m)\times_{U(n)} \Hom(G,U(n)) \to Gr(n,\C^m)\, .
\]
An element of the total space is precisely a pair $(V,f)$ where $V\subseteq\C^m$ is an $n$-dimensional linear subspace and $f\co G\to U(V)$ is a unitary representation of $G$ on it. If $m=\infty$, then the topology on the total space is again the colimit topology as the spaces $\Hom(G,U(n))$ are compact and Hausdorff.

\begin{construction} \label{cons:kdef}
The $\Gamma$-space
\[
\Kdef(G)_{\leq m}\co \Fin_\ast\to \Topp
\]
is defined on $\langle 1\rangle\in \Fin_\ast$ by
\[
\Kdef(G)_{\leq m}(\langle 1\rangle)= \bigsqcup_{n\geq 0} V(n,\C^m) \times_{U(n)} \Hom(G,U(n))
\]
and on $\langle k\rangle\in \Fin_\ast$ by the subspace
\[
\Kdef(G)_{\leq m}(\langle k\rangle)\subseteq\Kdef(G)(\langle 1\rangle)_{\leq m}^k
\]
consisting of those $k$-tuples $((V_1,f_1),\dots,(V_k,f_k))$ for which $V_i$ and $V_j$ are orthogonal to each other if $i\neq j$. If $\alpha\co \langle k\rangle \to \langle l\rangle$ is a morphism in $\Fin_\ast$, then
\[
\Kdef(G)_{\leq m}(\alpha)\co \Kdef(G)_{\leq m}(\langle k\rangle)\to \Kdef(G)_{\leq m}(\langle l\rangle)
\]
is defined by
\[
\Kdef(G)_{\leq m}(\alpha)((V_1,f_1),\dots,(V_k,f_k))=((W_1,g_1),\dots,(W_l,g_l))
\]
with $W_j=\bigoplus_{i \in \alpha^{-1}(j)} V_i$ and $g_j=\bigoplus_{i \in \alpha^{-1}(j)} f_i$ and where the direct sum of representations has the obvious meaning. (If $\alpha^{-1}(j)$ is empty, then $W_j=0$ and $f_j$ is the trivial homomorphism.)
\end{construction}

When $m=\infty$ we denote the $\Gamma$-space so defined by $\Kdef(G)$. This $\Gamma$-space is special; this can be seen by noting that the inclusion of the trivial group $1\hookrightarrow G$ induces a pullback square
\[
\xymatrix{
\Kdef(G)(\langle k\rangle) \ar@{^(->}[r] \ar[d] & \Kdef(G)(\langle 1\rangle)^k \ar[d] \\
\Kdef(1)(\langle k\rangle)  \ar@{^(->}[r]^-{\sim} & \Kdef(1)(\langle 1\rangle)^k
}
\]
in which the right map is a fibration. The bottom map is a bijection on components, and on each component it is an inclusion of the form
\[
X(i_1,\dots,i_k)\hookrightarrow Gr(i_1,\C^\infty)\times\cdots \times Gr(i_k,\C^\infty)
\]
for some $i_1,\dots,i_k\geq 0$, where $X(i_1,\dots,i_k)$ is the subspace of the product carved out by the condition of pairwise orthogonality. But this inclusion is a homotopy equivalence, in particular a weak equivalence, and hence the top map in the pullback square is also a weak equivalence. 

\begin{lemma} \label{lem:cofibrant}
$\Kdef(G)$ is cofibrant if $G$ is finitely generated.
\end{lemma}
\begin{proof}
This can be proved in the same way as \cite[Example 6.3.16]{Global} which, considered non-equivariantly, is the special case $\Kdef(1)$. For $0\leq k\leq m$ let $L(\C^k,\C^m)$ denote the space of linear isometries $\C^k\hookrightarrow \C^m$ equipped with the compact open topology. The only step that needs adaptation is the fact that for any dimension vector $(i_1,\dots,i_k)\in \N^k$ and $m\geq \sum_j i_j$ the product
\[
L(\C^{i_1}\oplus\cdots \oplus \C^{i_k},\C^m)\times \Hom(G,U(i_1))\times \cdots \times \Hom(G,U(i_k))
\]
admits an $(U(i_1)\times \cdots \times U(i_k))$-CW-structure for the diagonal action and is therefore equivariantly cofibrant. But this follows from the equivariant triangulation theorem of Park and Suh \cite[Theorem 1.3]{Park}; since $G$ is finitely generated, $\Hom(G,U(i_1)\times \cdots \times U(i_k))$ is a real algebraic variety on which $U(i_1)\times \cdots \times U(i_k)$ acts algebraically.
\end{proof}

There is an analogue for the conjugation quotient
\[
\Rep(G,U(n))=\Hom(G,U(n))/U(n)\, .
\]
The direct sum of representations gives $\bigsqcup_{n\geq 0} \Rep(G,U(n))$ the structure of an abelian topological monoid. An abelian monoid $M$ in $\Topp$ can always be viewed as a special $\Gamma$-space $M\co \Fin_\ast\to \Topp$ by setting $M(\langle k\rangle)=M^k$ and $M(\alpha)(m_1,\dots,m_k)=(n_1,\dots,n_l)$ on morphisms $\alpha\co \langle k\rangle \to \langle l\rangle$ where $n_j=\sum_{i\in \alpha^{-1}(j)} m_i$. (If $\alpha^{-1}(j)$ is empty, then $n_j$ is the neutral element.)

\begin{construction} \label{cons:rdef} We let
\[
\Rdef(G)\co \Fin_\ast\to \Topp
\]
be the $\Gamma$-space associated with the abelian monoid $\bigsqcup_{n\geq 0} \Rep(G,U(n))$.
\end{construction}

The latching maps
\[
\textnormal{colim}_{S\subsetneq \{1,\dots,n\}} \Rdef(G)(S_+)\to \Rdef(G)(\langle n\rangle)
\]
are iterated pushout-products of the basepoint inclusion $\ast=\Rep(G,U(0))\to \Rdef(G)(\langle 1\rangle)$. They are cofibrations if $\Rep(G,U(n))$ is cofibrant for every $n\geq 0$, and this follows again as in Lemma \ref{lem:cofibrant} from the equivariant triangulation theorem applied to $\Hom(G,U(n))$. In particular, we have

\begin{lemma}
$\Rdef(G)$ is cofibrant if $G$ is finitely generated.
\end{lemma}

Now both $\Kdef(G)_{\leq m}$ and $\Rdef(G)_{\leq m}$ can be upgraded to $\N^\otimes$-spaces;

\begin{lemma}
For $(x_1,\dots,x_k)\in \N^k$ we let $\Kdef(G)_{\leq m}(x_1,\dots,x_k) \subseteq \Kdef(G)_{\leq m}(\langle k\rangle)$ be the subspace of all $k$-tuples $((V_1,f_1),\dots,(V_k,f_k))$ with $\dim_{\C}(V_i)=x_i$ for all $i=1,\dots,k$.

Similarly, let $\Rdef(G)_{\leq m}(x_1,\dots,x_k)\subseteq \Rdef(G)_{\leq m}(\langle k\rangle)$ be the subspace of $k$-tuples of representations $([f_1],\dots,[f_k])$ with $\dim([f_i])=x_i$ for all $i=1,\dots,k$.

So defined, both $\Kdef(G)_{\leq m}$ and $\Rdef(G)_{\leq m}$ are $\N^\otimes$-spaces in a natural way.
\end{lemma}
\begin{proof}
This is clear.
\end{proof}

By composing with the singular complex and applying the nerve $\mathrm{N}(-)$, it follows that $\Kdef(G)$ and $\Rdef(G)$ define $\mathbb{E}_\infty$-algebras in $\N$-graded spaces. Abusing notation slightly, we shall denote these $\mathbb{E}_\infty$-algebras also by $\Kdef(G)$ and $\Rdef(G)$, respectively.

Let $\widehat{A}=\Hom(A,S^1)$ denote the Pontryagin dual of $A$.

\begin{proposition} \label{prop:key}
Let $G,A$ be finitely generated discrete groups and $A$ abelian. Then twisting $G$-representations by characters of $A$ induces equivalences of $\mathbb{E}_\infty$-algebras in $\Spc^\N$
\begin{enumerate}
\item[(1)] $\Kdef(G\times A)\simeq \Kdef(G) \otimes \widehat{A}$
\item[(2)] $\Rdef(G\times A)\simeq \Rdef(G) \otimes \widehat{A}$
\end{enumerate}
and these equivalences are natural in $G$ and $A$.
\end{proposition}
\begin{proof}
We shall only prove (1) as the proof of (2) is analogous.

We will prove an isomorphism of $\Gamma$-spaces
\[
\Kdef(G)_{\widehat{A}_+} \cong \Kdef(G\times A) 
\]
and subsequently observe that this isomorphism respects the graded structure. Thus, it induces an isomorphism of $\N^\otimes$-spaces, and this implies the claim by Lemma \ref{lem:singular} and Corollary \ref{cor:tensor2}.

Let $V\subseteq\C^\infty$ be a finite dimensional linear subspace. A homomorphism $f\co G\times A\to U(V)$ corresponds to a pair $(f',\phi)$ where $f'\co G\to U(V)$ and $\phi \co A\to U(V)$ are homomorphisms such that $f'(g)\phi(a)=\phi(a)f'(g)$ for all $(g,a)\in G\times A$. Since $A$ is abelian, $\phi$ splits $V$ into an orthogonal sum $\bigoplus_{i=1}^k V_i$ such that the action of $\phi$ on each summand is determined by a character $\phi_i\in \Hom(A,U(1))=\widehat{A}$, all $V_i$ are non-zero and the $\phi_i$ are mutually distinct. By commutativity of $f'$ and $\phi$, $f'$ restricts to representations $f_i'\co G\to U(V_i)$. This gives a bijective correspondence
\begin{equation} \label{eq:correspondence}
(V,f)\mapsto \{(V_i,f_i',\phi_i)\}_{i=1}^k\, .
\end{equation}

Fix $\langle l\rangle \in \Fin_\ast$ and $m\geq 0$ and define a homeomorphism
\[
g_{m,\langle l\rangle}\co \Kdef(G)_{\leq m}(\langle l\rangle\wedge \widehat{A}_+) \to \Kdef(G\times A)_{\leq m}(\langle l\rangle)\,.
\]
as follows. Recall that the left hand side is a quotient space of the coproduct
\[
\bigsqcup_{k\geq 0} \Kdef(G)_{\leq m}(\langle k\rangle) \times [(\langle l \rangle^\circ \times \widehat{A})_{+}]^k\,.
\]
Under the equivalence relation (\ref{eq:equivalencerelation}) every element in the coproduct is equivalent to one of the form
\[
((V_i,f'_i)_{i=1}^k,(n_i,\phi_i)_{i=1}^k),
\]
that is, to one where the extra basepoint of $(\langle l\rangle^\circ \times \widehat{A})_+$ does not appear. To such an element we associate
\[
(W_j,f_j)_{j=1}^l \in \Kdef(G\times A)_{\leq m}(\langle l\rangle)
\]
with $(W_j,f_j)$ defined by
\[
(W_j,f_j) \mapsto \{(V_i,f_i',\phi_i)\mid i: n_i=j\}
\]
under the correspondence (\ref{eq:correspondence}). This defines the continuous bijection $g_{m,\langle l\rangle}$; as the domain of $g_{m,\langle l\rangle}$ is compact and the codomain is Hausdorff, $g_{m,\langle l\rangle}$ is in fact a homeomorphism.

For $(x_1,\dots,x_l)\in \N^l$, the element $((V_i,f'_i)_{i=1}^k,(n_i,\phi_i)_{i=1}^k)$ defines an element of the subspace $\Kdef(G)_{\widehat{A}_+}(x_1,\dots,x_l)$ precisely if $\sum_{i: n_i=j} \dim_{\C}(V_i)=x_j$ for all $j=1,\dots,l$. It is therefore mapped to a tuple $(W_j,f_j)_{j=1}^l$ with $\dim_{\C}(W_j)=x_j$, which shows that the isomorphism is in fact one of $\N^\otimes$-spaces.

Taking the colimit over $m$, and noting that $ \Kdef(G)(\langle l\rangle \wedge \widehat{A}_+)$ carries the colimit topology as $\langle l \rangle \wedge \widehat{A}_+$ is compact Hausdorff, we hence obtain a homeomorphism
\[
g_{\langle l\rangle}\co \Kdef(G)(\langle l \rangle\wedge \widehat{A}_+) \to \Kdef(G\times A)(\langle l \rangle)\, .
\]
By inspection, this map is natural in $\langle l \rangle\in \Fin_\ast$ and thus defines an isomorphism of $\N^\otimes$-spaces. Evidently, this isomorphism is natural in $G$ and $A$.
\end{proof}

\begin{remark} \label{rem:productformula}
The equivalence $\Kdef(G\times A)\simeq \Kdef(G)\otimes \widehat{A}$ may be viewed as a special instance of Lawson's product formula \cite{Lawson} applied to $G\times A$, but one which holds \emph{before} passing to group completions. Indeed, after group completion we have an equivalence of $\mathbb{E}_\infty$-groups
\[
\Omega B\, \Kdef(G\times A) \simeq \Omega B\, \Kdef(G) \otimes \widehat{A}\,,
\]
which is precisely what the product formula predicts:
\[
k^{\mathrm{def}}(G\times A)\simeq k^{\mathrm{def}}(G)\wedge_{ku} k^{\mathrm{def}}(A) \simeq k^{\mathrm{def}}(G)\wedge \widehat{A}_+\, ,
\]
where $k^{\mathrm{def}}(G)$ is the connective spectrum corresponding to $\Omega B\, \Kdef(G)$.
\end{remark}

\begin{remark} \label{rem:pathconnected}
Implicit in the proof of Proposition \ref{prop:key} is the fact that if $A=\Z^r$, then the projection $G\times A\to G$ induces a bijection of path-components
\[
\pi_0 \Hom(G,U(n))\xrightarrow{\cong} \pi_0 \Hom(G\times \Z^r,U(n))
\]
for all $n\geq 0$.
\end{remark}

\begin{variant} \label{var:based}
There are sub-$\Gamma$-spaces $\Kdef_{\mathds{1}}(G)\subseteq \Kdef(G)$ and $\Rdef_{\mathds{1}}(G)\subseteq \Rdef(G)$ defined by using only the path-component of the trivial homomorphism
\[
\Hom_{\mathds{1}}(G,U(n))\subseteq \Hom(G,U(n))
\]
in Constructions \ref{cons:kdef} and \ref{cons:rdef}. The proof of Proposition \ref{prop:key} applies to show that for any finitely generated abelian group $A$ there are equivalences of $\mathbb{E}_\infty$-algebras in $\Spc^\N$
\begin{enumerate}
\item[(1)] $\Kdef_{\mathds{1}}(G\times A)\simeq \Kdef_{\mathds{1}}(G)\otimes \widehat{\mathrm{Free}(A)}$
\item[(2)] $\Rdef_{\mathds{1}}(G\times A)\simeq \Rdef_{\mathds{1}}(G)\otimes \widehat{\mathrm{Free}(A)}$
\end{enumerate}
where $\mathrm{Free}(A)=A/\mathrm{Tors}(A)\cong \Z^{\mathrm{rk}(A)}$ is the free part of $A$. These equivalences are natural in $G$ and $A$.
\end{variant}

Note that if $A$ is a finite abelian group, then $\Hom_{\mathds{1}}(A,U(n))$ consists of only the trivial homomorphism (see Remark \ref{rem:finitegroup}). Restriction to $A$ gives a map
\[
\Hom(G\times A,U(n))\to \Hom(A,U(n))
\]
which shows that $\Hom_{\mathds{1}}(G\times A,U(n))$ consists of homomorphisms which restrict to the trivial homomorphism on $A$. This explains why in the equivalences above only the free part of a finitely generated abelian group shows up.


\section{Homology stability} \label{sec:hs}

\subsection{Application of the local-to-global principle} \label{sec:hsproof}

For every $n\geq 1$, the natural embedding $U(n)\subseteq U(n+1)$ induces a map
\[
\Hom_{\mathds{1}}(G,U(n)) \to \Hom_{\mathds{1}}(G,U(n+1)) \,.
\]
The goal of this section is to prove the following theorem.

\begin{theorem} \label{thm:main1}
The sequence $\Hom_{\mathds{1}}(G,U(1)) \to \Hom_{\mathds{1}}(G,U(2)) \to \cdots$ satisfies homology stability for a finitely generated group $G$ if and only if it satisfies homology stability for $G\times \Z^r$ for every $r\geq 0$.
\end{theorem}

Theorem \ref{thm:main1} holds for arbitrary constant coefficients, for example $\Z$. For the analogous theorem concerning the representation varieties $\Rep(G,U(n))$ see Theorem \ref{thm:main2} below.

The following is useful terminology borrowed from \cite{KM18}.

\begin{definition} \label{def:chargedeinfinity}
A \emph{charged $\mathbb{E}_\infty$-space} is an $\mathbb{E}_\infty$-algebra in $\Spc^\N$. We say that a charged $\mathbb{E}_\infty$-space $A$ is \emph{connected} if its underlying space has $\pi_0(A)\cong \N$.
\end{definition}

Suppose that $A$ is a connected charged $\mathbb{E}_\infty$-space. Then the underlying space is $A\cong \bigsqcup_{n\geq 0} A_n$ with each $A_n$ connected. Pick any element $a_1\in A_1$; then multiplication by $a_1$ using the $\mathbb{E}_\infty$-structure gives a sequence of maps
\begin{equation} \label{eq:homologystability}
A_0\xrightarrow{\cdot a_1} A_1\xrightarrow{\cdot a_1} A_2\xrightarrow{\cdot a_1} \cdots
\end{equation}
Since $A_1$ is connected, the maps are independent, up to homotopy, of the specific choice of $a_1$.

\begin{definition}
The charged $\mathbb{E}_\infty$-space $A$ is said to satisfy \emph{homology stability} if the sequence (\ref{eq:homologystability}) satisfies homology stability.
\end{definition}

Let $A$ be a connected charged $\mathbb{E}_\infty$-space and let $X\in \Spc$ be a connected space. Then the charged $\mathbb{E}_\infty$-space $A\otimes X$ is again connected; one way to see this informally is to note first that the underlying space of $A\otimes X$ is $\bigsqcup_{n\in \N} A_{X_+}(n)$, where $A$ is viewed as an $\N^\otimes$-space. By the coend description of $A_{X_+}(n)$ in Section \ref{sec:graded}, a point in $A_{X_+}(n)$ is the equivalence class of a configuration of $k$ points in $X$ and a label in $A(y_1,\dots,y_k)$ for some $(y_1,\dots,y_k)\in \N^k$ with $\sum_{i} y_i=n$. Since $X$ is connected, the $k$ points can move continuously into any chosen point of $X$, showing that $\pi_0(A_{X_+}(n))\cong \pi_0(A(n))$. Since $A$ is connected, so is $A(n)$ for every $n$, and thus $\pi_0(A\otimes X)\cong \N$.

The local-to-global principle for homology stability of Kupers and Miller \cite{KM18} gives a relationship between homology stability of $A$ and homology stability of $A\otimes X$. We state the following corollary of their result.

\begin{theorem}[\cite{KM18}] \label{thm:localtoglobal}
If a connected charged $\mathbb{E}_\infty$-space $A$ satisfies homology stability, then so does $A\otimes M$ for any smooth connected orientable manifold $M$.
\end{theorem}

\begin{remark} To arrive at Theorem \ref{thm:localtoglobal} note the following:
\begin{enumerate}
\item[(1)] The results of $\cite{KM18}$ are formulated in terms of factorisation homology (or topological chiral homology) of framed $\mathbb{E}_n$-algebras over oriented $n$-manifolds $M$. But every $\mathbb{E}_\infty$-space $A$ may be viewed as a framed $\mathbb{E}_n$-algebra for any $n$ by neglect of structure. In this situation, by \cite[Proposition 5.1]{AF1} there is a canonical equivalence of spaces
\begin{equation*} \label{eq:fh}
\int_M A \simeq A \otimes M\,,
\end{equation*}
where on the left is factorisation homology and on the right is tensoring.
\item[(2)] The main result of \cite{KM18} requires the manifold $M$ to be non-compact, in order to be able to construct the stabilisation maps exhibiting homology stability of $\int_M A$. But if $A$ is an $\mathbb{E}_\infty$-algebra, then by the previous equivalence $\int_M A$ is homotopy invariant in $M$, and hence any given manifold can be replaced by a non-compact one simply by taking Cartesian product with $\R$.
\end{enumerate}
\end{remark}

We now apply Theorem \ref{thm:localtoglobal} to the charged $\mathbb{E}_\infty$-space $\Kdef_{\mathds{1}}(G)$. The underlying space of $\Kdef_{\mathds{1}}(G)$ is $\bigsqcup_{n\geq 0} \Hom_{\mathds{1}}(G,U(n))_{hU(n)}$. Because we have singled out the connected component of the trivial homomorphism, we see that $\pi_0(\Kdef_{\mathds{1}}(G))\cong \N$, i.e., $\Kdef_{\mathds{1}}(G)$ is a connected charged $\mathbb{E}_\infty$-space.

\begin{lemma} \label{lem:homologystability}
The charged $\mathbb{E}_\infty$-space $\Kdef_{\mathds{1}}(G)$ satisfies homology stability if and only if the sequence $\Hom_{\mathds{1}}(G,U(1)) \to \Hom_{\mathds{1}}(G,U(2)) \to \cdots$ satisfies homology stability.
\end{lemma}
\begin{proof}
By definition, $\Kdef_{\mathds{1}}(G)$ satisfies homology stability if so does the sequence
\[
\Hom_{\mathds{1}}(G,U(1))_{hU(1)} \xrightarrow{\oplus \mathds{1}} \Hom_{\mathds{1}}(G,U(2))_{hU(2)} \xrightarrow{\oplus \mathds{1}} \cdots
\]
where $\oplus \mathds{1}$ means direct sum with a trivial representation. These spaces fit into a sequence of homotopy fibre sequences
\[
\Hom_{\mathds{1}}(G,U(n))\to \Hom_{\mathds{1}}(G,U(n))_{hU(n)} \to BU(n)\,;
\]
because the sequence $BU(1)\to BU(2)\to \cdots$ satisfies homology stability, and $BU(n)$ is simply-connected for all $n$, the Serre spectral sequence and Zeeman's comparison theorem imply the result.
\end{proof}

\begin{proof}[Proof of Theorem \ref{thm:main1}]
By Lemma \ref{lem:homologystability}, it suffices to see that homology stability of $\Kdef_{\mathds{1}}(G)$ implies homology stability of $\Kdef_{\mathds{1}}(G\times \Z^r)$. By Variant \ref{var:based} of Proposition \ref{prop:key} there is an equivalence of charged $\mathbb{E}_\infty$-spaces
\[
\Kdef_{\mathds{1}}(G\times \Z^r) \simeq \Kdef_{\mathds{1}}(G) \otimes (S^1)^r\, .
\]
The claim follows now from Theorem \ref{thm:localtoglobal}.
\end{proof}

The proof of the following variant is almost identical except for replacing $\Kdef_{\mathds{1}}(G)$ by $\Rdef_{\mathds{1}}(G)$ throughout and omitting the argument with the Serre spectral sequence.

\begin{theorem} \label{thm:main2}
The sequence $\Rep_{\mathds{1}}(G,U(1)) \to \Rep_{\mathds{1}}(G,U(2)) \to \cdots$ satisfies homology stability for a finitely generated group $G$ if and only if it satisfies homology stability for $G\times \Z^r$ for every $r\geq 0$.
\end{theorem}

\begin{remark} \label{rem:finitegroup} If $A$ is a non-trivial \emph{finite} abelian group, then $\Hom(G\times A,U(n))$ fails to satisfy homology stability already on $H_0$.  The isomorphism class of a representation $f\co A\to U(n)$ is determined by an unordered tuple $\chi=\{\chi_1,\dots,\chi_n\}$ of $n$ $1$-dimensional characters of $A$; there is a decomposition
\[
\Hom(A,U(n))\cong \bigsqcup_{\chi \in \widehat{A}^n/\Sigma_n} U(n)/U(n)_{\chi}
\]
where $U(n)_\chi$ is the stabiliser of the representation $\chi_1\oplus \cdots \oplus \chi_n$ under the conjugation action, see \cite[Proposition 2.5]{AG}. Consequently, 
\[
|\pi_0\Hom(G\times A,U(n))|\geq |\pi_0\Hom(A,U(n))| = \sum_{k=1}^{|A|} \binom{n-1}{k-1} \binom{|A|}{k}
\]
which diverges as $n$ tends to infinity. (The inequality holds, because one space is a retract of the other.) If we restrict to only the path-component of the trivial representation, then the decomposition above implies
\[
\Hom_{\mathds{1}}(G\times A,U(n)) \cong \Hom_{\mathds{1}}(G,U(n))\,,
\]
induced by restriction to $G$. So, for the question of homology stability direct products with finite abelian groups do not really play a role.
\end{remark}

\subsection{An example} \label{sec:examples}

Let $\Sigma$ be a closed orientable surface of genus $g>0$. In his thesis \cite[Section 4.5]{RThesis} Ramras uses gauge theoretic methods to show that the maps
\[
\Hom(\pi_1(\Sigma),U(n)) \to \Hom(\pi_1(\Sigma),U(n+1))
\]
are precisely $(2n-1)$-connected. He also showed the equivalence
\begin{equation} \label{eq:surfacestable}
\Hom(\pi_1(\Sigma),U)\simeq \Omega^\infty(ku\wedge B\pi_1(\Sigma))\, ,
\end{equation}
where $\Hom(\pi_1(\Sigma),U)=\colim_n\, \Hom(\pi_1(\Sigma),U(n))$, $ku$ is the connective complex $K$-theory spectrum and $\Omega^\infty$ takes the underlying space of a spectrum.

Combining this with Theorem \ref{thm:main1} we obtain:

\begin{corollary}
Let $\mathcal{C}$ be the smallest class of groups which contains the fundamental groups of closed orientable surfaces and which is closed under binary free products and direct product with $\Z$. Then, for every $G\in \mathcal{C}$ the sequence
\[
\Hom(G,U(1))\to \Hom(G,U(2)) \to \cdots
\]
satisfies homology stability with integer coefficients. Moreover, the stable homology is known in view of the equivalence
\[
\Hom(G,U)\simeq \Omega^\infty(ku\wedge BG)\, .
\]
\end{corollary}
\begin{proof}
In view of the introductory remarks, it suffices to see that the property of homology stability is preserved under binary free products and direct product with $\Z$. The latter follows from Theorem \ref{thm:main1}, while the former follows from the K{\"u}nneth exact sequence and the evident homeomorphism
\begin{equation} \label{eq:freeproduct}
\Hom(\Gamma_1\ast \Gamma_2,U(n)) \cong \Hom(\Gamma_1,U(n)) \times \Hom(\Gamma_2,U(n))
\end{equation}
for any two discrete groups $\Gamma_1, \Gamma_2$. Also note that $\Hom(\pi_1(\Sigma),U(n))$ is path-connected by \cite[Corollary 4.11]{RYangMills}, and hence so is $\Hom(G,U(n))$ (see Remark \ref{rem:pathconnected}); therefore, no subscript $\mathds{1}$ is needed.

Because of (\ref{eq:freeproduct}) and the equivalence
\[
\Omega^\infty(ku\wedge B(\Gamma_1\ast \Gamma_2)) \simeq \Omega^\infty(ku\wedge B\Gamma_1)\times  \Omega^\infty(ku\wedge B\Gamma_2)\, ,
\]
it suffices to prove the second statement when $G=\pi_1(\Sigma)\times \Z^k$. Since $\Hom(G,U(n))$ is path-connected for all $n$, $\Kdef(G)$ is stably group-like (in the sense of \cite[Section 2.3]{RYangMills}) with respect to the trivial representation; it follows then from \cite[Theorem 2.3]{RYangMills} that $\Hom(G,U)$ is the homotopy fibre over the basepoint of the canonical map $\Omega B \Kdef(G)\to \Omega B \Kdef(1)$. By Remark \ref{rem:productformula}, or Lawson's product formula \cite{Lawson}, and (\ref{eq:surfacestable})
\[
\Omega B \Kdef(G) \simeq \Omega^\infty(ku\wedge BG_+)\,,
\]
and the result follows.
\end{proof}

Because $BG$ for $G\in \mathcal{C}$ is stably a wedge of spheres, $\Hom(G,U)$ is a finite product of various higher connected covers of $BU$ and $U$. The homology of these connected covers was computed by Singer \cite{Singer} and Stong \cite{Stong}.

Ramras \cite{RThesis} also showed that if $\mathcal{N}$ is a closed aspherical non-orientable surface, then
\[
\Hom_{\mathds{1}}(\pi_1(\mathcal{N}),U(n)) \to \Hom_{\mathds{1}}(\pi_1(\mathcal{N}),U(n+1))
\]
is at least $(n-1)$-connected. Let $\mathcal{C}'$ be the class of groups obtained from $\pi_1(\mathcal{N})$ by taking binary free products and direct products with $\Z$. With the same proof as before one obtains homology stability for $\Hom_{\mathds{1}}(G,U(n))$ as $n\to \infty$ for any $G\in \mathcal{C}'$.

By \cite[Corollary 4.12]{RYangMills} the space $\Hom(\pi_1(\mathcal{N}),U(n))$ has two path components when $n\geq 1$. Using \cite[Theorem 6.1]{RModuli}, for $G\in \mathcal{C}'$ the space $\Hom_{\mathds{1}}(G,U)$ has an explicit, yet slightly more cumbersome description in terms of higher connected covers of $BU$, $U$ and $\Omega^\infty \textnormal{cofib}(ku\xrightarrow{2} ku)$.

\begin{remark}
The theorem of Kupers--Miller, Theorem \ref{thm:localtoglobal}, has a version with explicit stability ranges (see \cite[Theorem 74]{KM18}). However, the ranges we'd obtain this way are in general not very close to the optimal range. For example, the map $\Hom(\Z^2,U(n))\to \Hom(\Z^2,U(n+1))$ is an isomorphism in homology degrees $i\leq 2n-2$, but the range we obtain from \cite[Theorem 74]{KM18} (choosing $\rho(n)=n/2$ in their notation) applied to our setting with $G=1$ would be $i\leq n/2$.
\end{remark}


\section{Hochschild homology and a spectral sequence} \label{sec:ss}

Let $G,A$ be finitely generated groups, with $A$ abelian. In this section we derive from Proposition \ref{prop:key} a spectral sequence which gives a method of computing the homology of $\Rep(G\times A,U(n))$ from knowledge of the homology of $\Rep(G,U(k))$ for $k\leq n$ along with the multiplicative structure on homology induced by the maps
\[
\Rep(G,U(k_1))\times \Rep(G,U(k_2))\xrightarrow{\oplus} \Rep(G,U(k_1+k_2))\,.
\]
There is a similar spectral sequence for the equivariant homology of $\Hom(G\times A,U(n))$.

Our reference for the following notions from higher algebra is \cite{HA}.

\begin{notation}
Let $k$ be a commutative ring spectrum. (We will only be concerned with the case where $k$ is an ordinary discrete commutative ring.)
\begin{enumerate}
\item[(1)] We denote by $\Mod_k$ the symmetric monoidal $\infty$-category of $k$-module spectra. The symmetric monoidal product is smash product over $k$, denoted $\otimes_k$. We will refer to objects of $\Mod_k$ as $k$-modules.
\item[(2)]  We denote by $\CAlg_k=\CAlg(\Mod_k)$ the $\infty$-category of $\mathbb{E}_\infty$-$k$-algebras.
\item[(3)] Since $\CAlg_k$ is cocomplete, it is canonically tensored over the $\infty$-category of spaces $\Spc$ and we denote by $\otimes$ the tensoring. Thus, for $X\in \Spc$ and $A\in \CAlg_k$, the object $A\otimes X\in \CAlg_k$ is determined by the demanding that $A\otimes -$ preserves colimits and that $A\otimes \ast\simeq A$ where $\ast\in \Spc$ is the point.
\item[(4)] Let $\mathbb{S}$ be the sphere spectrum, $\otimes_{\mathbb{S}}$ the smash product over $\mathbb{S}$, and for $X\in \Spc$ let $\mathbb{S}[X]$ be the suspension spectrum of $X_+$. We then have a functor
\[
k[-]\co \Spc\to \Mod_k,\quad X\mapsto \mathbb{S}[X]\otimes_{\mathbb{S}} k\, .
\]
In case that $k$ is a discrete commutative ring, the homotopy groups of $k[X]$ are the singular homology groups of $X$ with coefficients in $k$.
\end{enumerate}
\end{notation}

\begin{remark}
If $k$ is a discrete commutative ring, then $\Mod_k$ is equivalent as a symmetric monoidal $\infty$-category to the derived $\infty$-category of $k$-modules $\mathcal{D}(k)$ equipped with the $k$-linear tensor product. We may think of $\mathcal{D}(k)$ as the underlying $\infty$-category of a model category of chain complexes of $k$-modules. Then we see that every ordinary differential graded commutative $k$-algebra $A$ (henceforth called just \emph{dgca}) gives rise to an object of $\CAlg_k$. In particular, the homotopy groups of $A$ (when $A$ is thought of as an object of $\CAlg_k$) coincide with the homology groups of $A$ (when $A$ is viewed as a dgca).
\end{remark}

\begin{definition} \label{def:hh}
Let $A\in \CAlg_k$ and $X\in \Spc$. The \emph{Hochschild homology of $A$ over $X$} is the commutative $k$-algebra $\HH_\ast^X(A):=\pi_\ast(A\otimes X)$.
\end{definition}

\begin{remark}  \label{rem:hochschild}
Let $k$ be discrete and $A$ a dgca over $k$, viewed as an object of $\CAlg_k$. Let $X$ be a pointed simplicial set. Then there is a hands-on dgca model for $A\otimes X$, often written as $\mathcal{L}(A,A)(X)$ (called the \emph{Loday construction}), see \cite[Section 5.1]{Pirashvili}. There is then an isomorphism of $k$-algebras $\HH^X_\ast(A)\cong H_\ast(\mathcal{L}(A,A)(X))$. In Definition \ref{def:hochschildcomplex} we will review this construction when $X=S^1$; we write $C(A)=\mathcal{L}(A,A)(S^1)$. If $A$ itself is discrete, then $C(A)$ is the usual Hochschild complex as defined in \cite[Section 1.1.3]{Loday}.
\end{remark}

\begin{proposition} \label{prop:ss}
Let $G$, $A$ be finitely generated discrete groups and $A$ abelian. Let $k$ be a discrete commutative ring such that $H_\ast(\Kdef(G);k)$ is flat as a $k$-module. Then there is a first quadrant spectral sequence
\[
E^2=\HH^{\widehat{A}}_\ast(H_\ast(\Kdef(G);k)) \Longrightarrow  H_\ast(\Kdef(G\times A);k)\, .
\]
The analogous statement holds with $\Kdef(G)$ replaced by $\Rdef(G)$.
\end{proposition}
\begin{proof}
Pick a finite simplicial model for $\widehat{A}$ and identify $\Kdef(G) \otimes \widehat{A}$ with the value of the $\Gamma$-space $\Kdef(G)$ on $(\widehat{A}_\bullet)_+$. This is a Reedy cofibrant simplicial space, because $\Kdef(G)$ is cofibrant (see Lemma \ref{lem:cofibrant}), and
\[
|\Kdef(G)((\widehat{A}_\bullet)_+)|\simeq \Kdef(G\times A)
\]
by Proposition \ref{prop:key}. There is a spectral sequence converging to the homology of the geometric realisation of a Reedy cofibrant simplicial space, see \cite[Theorem 11.14]{Mayiteratedloopspaces}. In our case it takes the form
\[
E^2=H_\ast(H_\ast(\Kdef(G)((\widehat{A}_\bullet)_+);k)) \Longrightarrow H_\ast(\Kdef(G\times A);k)\, ,
\]
where $H_\ast(\Kdef(G)((\widehat{A}_\bullet)_+);k)$ is the chain complex of the simplicial $k$-module
\[
[n]\mapsto H_\ast(\Kdef(G)((\widehat{A}_n)_+);k)\,.
\]
Since $H_\ast(\Kdef(G);k)$ is flat over $k$, and using the fact that $\Kdef(G)$ is special, the K{\"u}nneth theorem implies that 
\[
H_\ast(\Kdef(G)((\widehat{A}_n)_+);k)\cong \bigotimes_{\widehat{A}_n} H_\ast(\Kdef(G);k)
\]
for every $n\geq 0$, functorially in $\widehat{A}_n$. This identifies the chain complex
\[
H_\ast(\Kdef(G)((\widehat{A}_\bullet)_+);k) 
\]
with the Hochschild complex computing $\HH^{\widehat{A}}_\ast(H_\ast(\Kdef(G);k))$, as defined for example in \cite[Section 5.1]{Pirashvili} (see Remark \ref{rem:hochschild}).

The statement for $\Rdef(G)$ is proved in the same way.
\end{proof}

\begin{remark}
One can show that the spectral sequences of Proposition \ref{prop:ss} are spectral sequences of algebras; since we will not use the multiplicative structure, we decided not to spell out the details. 
\end{remark}

Thanks to the graded structure on $\Kdef(G)$ and $\Rdef(G)$ we also obtain spectral sequences converging to the homology of $\Hom(G\times A,U(n))_{hU(n)}$ and $\Rep(G\times A,U(n))$ for each individual $n$. To this end observe that $H_\ast(\Kdef(G);k)$ is a \emph{bigraded} algebra, with an internal homological degree and a grading by the rank of a representation, that is, through the decomposition $\Kdef(G)\simeq \bigsqcup_{n\geq 0} \Hom(G,U(n))_{hU(n)}$. In Section \ref{sec:example} we will term this grading as the grading by \emph{charge} (as opposed to the homological grading). The grading by charge induces a grading on the Hochschild homology and we denote by
\[
\HH^{\widehat{A}}(H_\ast(\Kdef(G);k))_n
\]
the subspace of elements of charge $n$. This will be explained in more detail in Section \ref{sec:example}.

\begin{corollary} \label{cor:ss}
Under the hypotheses of Proposition \ref{prop:ss} there are spectral sequences for all $n\geq 0$
\[
E^2=\HH^{\widehat{A}}(H_\ast(\Kdef(G);k))_n \Longrightarrow H_\ast^{U(n)}(\Hom(G\times A,U(n));k)\,,
\]
and similarly for $\Rdef(G)$ and $\Rep(G\times A,U(n))$.
\end{corollary}
\begin{proof}
Because $\Kdef(G)$ comes from a $\N^\otimes$-space, there is an isomorphism of simplicial spaces
\[
\Kdef(G)((\widehat{A}_\bullet)_+) = \bigsqcup_{n\geq 0} \Kdef(G)((\widehat{A}_\bullet)_+)_n
\]
such that
\[
|\Kdef(G)((\widehat{A}_\bullet)_+)_n| \simeq \Hom(G\times A,U(n))_{hU(n)}\,.
\]
Therefore, the spectral sequence of Proposition \ref{prop:ss} is in fact a direct sum of spectral sequences, one for each $n$.
\end{proof}

\begin{remark} \label{rem:naturality}
We observe that the quotient maps $\Hom(G,U(n))_{hU(n)} \to \Rep(G,U(n))$ induce a map of $\Gamma$-spaces $\Kdef(G)\to \Rdef(G)$, and therefore a map of the corresponding spectral sequences in Proposition \ref{prop:ss} and Corollary \ref{cor:ss}. 
\end{remark}

An application of the spectral sequence will be given in Theorem \ref{thm:poincarerep}, where we determine the Poincar{\'e} polynomial of $\Rep(F_2\times \Z^r,U(2))$ for every $r\geq 0$.


\section{Free group representations and formality} \label{sec:formality}

Our next goal is to understand the homology $H_\ast(\Kdef(G);k)$ when $G$ is the direct product of a free group with an abelian group and when $k$ is a field of characteristic zero.

\subsection{Collapse of the spectral sequence} \label{sec:hh}

Let $G,A$ be finitely generated groups and $A$ abelian. Since $k[-]$ preserves colimits, Proposition \ref{prop:key} implies an equivalence of $\mathbb{E}_\infty$-$k$-algebras
\[
k[\Kdef(G\times A)] \simeq k[\Kdef(G)]\otimes \widehat{A}\, .
\]
If $k[\Kdef(G)]$ happens to be equivalent as an $\mathbb{E}_\infty$-$k$-algebra to
\[
\pi_\ast(k[\Kdef(G)]) \cong H_\ast(\Kdef(G);k)\,,
\]
then
\[
H_\ast(\Kdef(G\times A);k) \cong \HH^{\widehat{A}}(H_\ast(\Kdef(G);k))\, ;
\]
put differently, the spectral sequence of Proposition \ref{prop:ss} collapses in this case and there are no multiplicative extensions. We will show in Section \ref{sec:formality} that this is indeed the case when $G$ is a free group and $k$ is a characteristic zero field, which leads to the following theorem.

\begin{theorem} \label{thm:hh}
Let $k$ be a field of characteristic zero. Let $A$ be a finitely generated abelian group and let $s\geq 0$. There is an isomorphism of graded commutative $k$-algebras
\[
H_\ast(\Kdef(F_s\times A);k)\cong \HH^{\widehat{A}}_\ast(H_\ast(\Kdef(F_s);k))\, .
\]
\end{theorem}

Using the graded structure we obtain, just like in Section \ref{sec:ss}

\begin{corollary} \label{cor:hh}
Let $k$ be a field of characteristic zero and $A$ a finitely generated abelian group. Then, for every $n\geq 0$ there is an isomorphism
\[
H_\ast^{U(n)}(\Hom(F_s\times A,U(n));k)\cong \HH^{\widehat{A}}_{\ast}(H_\ast(\Kdef(F_s);k))_n\, ,
\]
where the subscript $n$ takes the subspace of elements of charge $n$.
\end{corollary}

The case $n=2$ will be discussed in detail in Section \ref{sec:example}.

\subsection{Homology of $\Kdef(F_s)$} \label{sec:homologyring}

We describe the algebra $H_\ast(\Kdef(F_s);k)$ appearing in Theorem \ref{thm:hh}.

We begin by describing the space underlying $\Kdef(F_s)$. Let $\Vect_{\C}^{\simeq}$ denote the space $\bigsqcup_{n\geq 0} BU(n)$ with its usual $\mathbb{E}_\infty$-structure. It is naturally a charged $\mathbb{E}_\infty$-space, and there is a graded equivalence $\Vect_{\C}^\simeq \simeq \Kdef(1)$.

The mapping space $\Map(\vee^s S^1,\Vect_{\C}^\simeq)$ is given an induced $\mathbb{E}_\infty$-structure as follows.

\begin{construction} \label{cons:map} Let $M$ be an $\mathbb{E}_\infty$-algebra in $\Spc^\N$. 
\begin{itemize}
\item Define $ev_1 \co \Map(S^1,M) \to M$ by the pullback square in $\CAlg(\Spc^\N)$
\[
\begin{tikzcd}
\Map(S^1,M) \arrow[r] \arrow[d, "ev_1"] &M \arrow[d, "\Delta"] \\
M \arrow[r, "\Delta"] & M\times M
\end{tikzcd}
\]
where $\Delta$ is the diagonal map.
\item For $s\geq 2$ define $ev_1 \co \Map(\vee^s S^1,M) \to M$ by the pullback square in $\CAlg(\Spc^\N)$
\begin{equation*} \label{dgr:freeloopspace}
\begin{tikzcd}
\Map(\vee^s S^1,M) \arrow[r] \arrow[d] & \Map(\vee^{s-1} S^1,M) \arrow[d, "ev_1"]  \\
\Map(S^1,M) \arrow[r, "ev_1"] & M\, .
\end{tikzcd}
\end{equation*}
\end{itemize}
\end{construction}

View $\Kdef(F_s)$ as an $\mathbb{E}_\infty$-algebra over $\Kdef(1)\simeq \Vect^\simeq_{\C}$ via the unique homomorphism $1\to F_s$.

\begin{proposition}[{cf. \cite[Section 5]{LBott}}] \label{prop:identificationkdef}
For every $s\geq 0$, there is a canonical equivalence
\[
\Kdef(F_s) \simeq \Map(\vee^s S^1,\mathrm{Vect}_{\C}^\simeq)
\]
as charged $\mathbb{E}_\infty$-spaces over $\Vect_{\C}^\simeq$.
\end{proposition}
\begin{proof}
The case $s=0$ is clear, so let us assume $s\geq 1$. It will be enough to show that $\Kdef(F_s)$ fits into the same kind of pullback squares as $\Map(\vee^s S^1,\Vect_{\C}^{\simeq})$ does. The homomorphism $i\co 1\to F_1=\Z$ induces a commutative diagram of charged $\mathbb{E}_\infty$-spaces
\begin{equation} \label{dgr:kdefz}
\begin{tikzcd}
\Kdef(\Z) \arrow[r, "i^\ast"] \arrow[d, "i^\ast"] & \Kdef(1) \arrow[d, "\Delta"] \\
\Kdef(1) \arrow[r, "\Delta"] & \Kdef(1) \times \Kdef(1)
\end{tikzcd}
\end{equation}
which we claim is a pullback diagram. Since the forgetful functor $\CAlg(\Spc^\N)\to \Spc$ reflects pullbacks, this can be checked on underlying spaces. On underlying spaces the map $i^\ast \co \Kdef(\Z)\to \Kdef(1)$ corresponds to the projection in the Borel construction
\[
i^\ast \co \bigsqcup_{n\geq 0} \Hom(\Z,U(n))_{hU(n)} \to \bigsqcup_{n\geq 0} BU(n)\, .
\]
The $U(n)$-equivariant equivalence (induced by functoriality of the bar-construction)
\[
\Hom(\Z,U(n))\xrightarrow{\sim} \Map_\ast(B\Z,BU(n))\simeq \Map_\ast(S^1,BU(n))\,,
\]
and the equivalence over $BU(n)$
\[
\Map_\ast(S^1,BU(n))_{hU(n)} \simeq \Map(S^1,BU(n))\,,
\]
induce an equivalence over $\bigsqcup_{n\geq 0} BU(n)$
\[
\begin{tikzcd}
{\displaystyle \bigsqcup_{n\geq 0} \Hom(\Z,U(n))_{hU(n)}} \arrow[rr, "\sim"] \arrow[dr, "i^\ast" swap] & & {\displaystyle \Map(S^1, \bigsqcup_{n\geq 0} BU(n))} \arrow[dl, "ev"] \\
& {\displaystyle \bigsqcup_{n\geq 0}BU(n)} &
\end{tikzcd}
\]
This equivalence shows that (\ref{dgr:kdefz}) is indeed a pullback square.

The homomorphisms $\Z\to F_s$ and $F_{s-1}\to F_s$ induced by the inclusion of the first and the last $s-1$ generators, respectively, induce a commutative diagram
\[
\begin{tikzcd}
\Kdef(F_s) \arrow[r] \arrow[d] & \Kdef(F_{s-1}) \arrow[d] \\
\Kdef(\Z)\arrow[r] &  \Kdef(1)
\end{tikzcd}
\]
which we claim is also a pullback diagram. Again this can be checked on underlying spaces where it follows from the homotopy pullback diagrams
\[
\begin{tikzcd}
\Hom(F_s,U(n)) \arrow[r] \arrow[d] & \Hom(F_{s-1},U(n)) \arrow[d] \\
\Hom(\Z,U(n)) \arrow[r] & \ast 
\end{tikzcd}
\]
after taking $U(n)$-homotopy orbits and disjoint union over all $n\geq 0$.
\end{proof}

Next we describe the cohomology of $\Map(\vee^s S^1,BU(n))$. Let $k$ be a commutative ring. For a space $X$ and an integer $g\in\{1,\dots,s\}$ we let
\[
ev|_g \co S^1\times \Map(\vee^s S^1,X) \to X
\]
denote evaluation along the $g$-th circle. Pick a fundamental class $e\in H_1(S^1;k)$. 

\begin{definition}
We define the  \emph{$g$-th free suspension} (cf. \cite{KK})
\begin{alignat*}{1}
\sigma_g\co H^\ast(X;k) & \to H^{\ast-1}(\Map(\vee^s S^1,X);k) \\
a & \mapsto (ev|_g)^\ast(a)/e
\end{alignat*}
where $-/e$ is the slant product.
\end{definition}

Consider the case $s=1$ and let $\sigma\co H^p(X;k)\to H^{p-1}(\Map(S^1,X);k)$ be the free suspension. Let $i\co \Omega X\to \Map(S^1,X)$ be the inclusion of the based into the free loop space. There is a commutative diagram
\[
\xymatrix{
H^p(X;k) \ar[r] \ar[dr]_-{ev^\ast} & H^p(S^1\times \Omega X;k) \ar[r]^-{-/e} & H^{p-1}(\Omega X;k) \\
& H^p(S^1\times \Map(S^1,X);k) \ar[u]_-{(id\times i)^\ast} \ar[r]^-{-/e} & H^{p-1}(\Map(S^1,X);k) \ar[u]_-{i^\ast} \, .
}
\]
Observe that if $p\geq 2$, then the composition of the maps in the top row can be identified, modulo sign, with the composition
\[
H^p(X;k)\to H^p(\Sigma \Omega X;k) \cong H^{p-1}(\Omega X;k)
\]
where the first map uses the counit of the loop-suspension-adjunction and the second map is the suspension isomorphism. Assume for simplicity that $X$ is simply-connected. Then this composition can also be identified, modulo sign, with the composition
\[
H^p(X;k) \cong E_2^{p,0} \to E_p^{p,0} \xrightarrow{d_p^{-1}} E_p^{0,p-1} \to E_2^{0,p-1} = H^{p-1}(\Omega X;k)\, ,
\]
where $(E_r,d_r)$ is the Serre spectral sequence of the path-loop fibration $\Omega X\to PX \to X$.

Consider the case $X=BU(n)$. Let us write $H^\ast(U(n);k)\cong \bigwedge_k(a_1,\dots,a_n)$ where $|a_j|=2j-1$. It is well-known that the generators may be chosen so that in the universal $U(n)$-bundle the class $a_j$ transgresses to the $j$-th Chern class. In other words, $d_{2j}^{-1}(c_j)=a_j$, and by changing the sign of $a_j$ if necessary, we get
\begin{equation} \label{eq:lerayhirsch}
i^\ast(\sigma(c_j))=a_j\, .
\end{equation}

Now let $s\geq 1$. For $g=1,\dots,s$ let
\[
\theta_g\co H^\ast(U(n);k)\to H^\ast(\Map(\vee^s S^1, BU(n));k)
\]
be the map defined by $\theta_g(a_j)=\sigma_g (c_j)$. Let
\[
ev_1\co \Map(\vee^s S^1,BU(n))\to BU(n)
\]
be evaluation at the basepoint of $\vee^s S^1$. Then $\theta_1,\dots,\theta_s$ and $ev_1$ induce a map of $H^\ast(BU(n);k)$-modules
\[
\theta\co H^\ast(U(n)^s;k)\otimes H^\ast(BU(n);k) \to H^\ast(\Map(\vee^s S^1,BU(n));k)
\]

\begin{lemma} \label{lem:LBUn}
The map $\theta$ is an isomorphism.
\end{lemma}
\begin{proof}
This is an application of the Leray-Hirsch theorem to the fibration sequence
\[
U(n)^s \xrightarrow{i} \Map(\vee^s S^1,BU(n)) \xrightarrow{ev_1} BU(n)\, .
\]
One uses (\ref{eq:lerayhirsch}) to show that the map $H^\ast(U(n)^s;k) \to H^\ast(\Map(\vee^s S^1,BU(n));k)$ which is built using $\theta_1,\dots,\theta_s$, is a section of $i^\ast$, the restriction to the fibre.
\end{proof}

Let $\mu_{m,n}\co \Map(\vee^s S^1,BU(m))\times \Map(\vee^s S^1,BU(n))\to \Map(\vee^s S^1,BU(m+n))$ be the product.

\begin{lemma} \label{lem:comultiplication}
In cohomology, $\mu_{m,n}^\ast(\sigma_g c_k)=\sum_{\substack{i+j=k\\ i\leq m,\, j\leq n}} (\sigma_g c_i\otimes c_j+c_i\otimes \sigma_g c_j)$.
\end{lemma}
\begin{proof}
This follows from a diagram chase. More precisely, we have that the two maps
\[
S^1\times \Map(\vee^s S^1,BU(m))\times \Map(\vee^s S^1,BU(n))\to BU(m+n)
\]
given by
\[
ev|_g\circ (id_{S^1}\times \mu_{m,n})
\]
and
\[
\mu_{m,n}\circ (ev|_g\times ev|_g)\circ \tau\circ (\Delta_{S^1}\times id)\,,
\]
are equal, where $\tau$ swaps the two factors in the middle and $\Delta_{S^1}$ is the diagonal on $S^1$. The claim follows from this identity by using naturality of the slant product.
\end{proof}

\begin{lemma} \label{lem:injectivecohomology}
The map $\mu_{n-1,1}$ is injective in cohomology.
\end{lemma}
\begin{proof}
Let $x\in H^\ast(\Map(\vee^s S^1,BU(n));k)$ and assume that $\mu_{n-1,1}^\ast(x)=0$. By Lemma \ref{lem:LBUn}, $x$ can be written as a linear combination of monomials of the form
\[
c_1^{i_1}\cdots c_n^{i_n} \sigma_1(c_1)^{\varepsilon_{11}}\cdots \sigma_1(c_n)^{\varepsilon_{1n}} \cdots \sigma_s(c_1)^{\varepsilon_{s1}}\cdots \sigma_s(c_n)^{\varepsilon_{sn}}
\]
with exponents $i_j\geq 0$ and $\varepsilon_{ij}\in \{0,1\}$. To each non-zero monomial we can assign a tuple of integers $(i_n,j_1,\dots,j_k)$ where the $1\leq j_1 < \cdots < j_k \leq s$ are precisely those indices for which $\varepsilon_{j_l n}=1$.  In other words, the monomial can be written as a product $c_n^{i_n} \sigma_{j_1}(c_n)\cdots \sigma_{j_k}(c_n) q$, where $q$ contains no terms that depend on $c_n$. Let $\mathcal{I}(x)$ be the collection of all the tuples that arise in this way from monomials appearing in $x$.

Suppose that $x$ is non-zero. Then $\mathcal{I}(x)$ is non-empty and there exists an element $(l,j_1,\dots,j_k)\in \mathcal{I}(x)$ for which $(k,l)$ is minimal in the lexicographic order on $\N^2$. We may write
\[
x=c_n^l  \sigma_{j_1}(c_n)\cdots \sigma_{j_k}(c_n) q+ x'
\]
where $q$ is non-zero and contains no terms that depend on $c_n$, and where $x'$ is a sum of monomials each of which contains either a power of $c_n$ greater than $l$, or a factor $\sigma_{r}(c_n)$ for $r\not\in \{j_1,\dots,j_k\}$. Recall that
\[
H^\ast(\Map(\vee^s S^1,BU(1));k)\cong k[c_1]\otimes {\textstyle \bigwedge_k}(\sigma_1(c_1),\dots,\sigma_s(c_1))
\]
and let $\pi$ be the projection onto the summand generated by $c_1^l \sigma_{j_1}(c_1)\cdots \sigma_{j_k}(c_1)$. Using Lemma \ref{lem:comultiplication} we get
\[
(id\otimes \pi)(\mu_{n-1,1}^\ast(x))=c_{n-1}^{l+k} q \otimes c_1^l \sigma_{j_1}(c_1)\cdots \sigma_{j_k}(c_1)\, .
\]
For this to be zero, we must have $q=0$, which is a contradiction.

Thus, $x=0$ and $\mu_{n-1,1}^\ast$ is injective.
\end{proof}

Since Lemma \ref{lem:injectivecohomology} holds for any commutative ring $k$, in particular for every field, and since the integral cohomology of $\Map(\vee^s S^1,\Vect_{\C}^{\simeq})$ is free, it follows that dually $\mu_{n-1,1}$ is surjective in homology. By induction, we obtain:

\begin{corollary} \label{cor:generation}
As a $k$-algebra, $H_\ast(\Map(\vee^s S^1,\Vect_{\C}^{\simeq});k)$ is generated by the homology of the disjoint summand $\Map(\vee^s S^1,BU(1))\subseteq \Map(\vee^s S^1,\Vect_{\C}^{\simeq})$.
\end{corollary}

Using the $H$-space structure on $\Map(\vee^s S^1,BU(1))$ induced by that of $BU(1)$ we get an isomorphism of $k$-algebras
\begin{equation} \label{eq:homologydim1}
H_\ast(\Map(\vee^s S^1,BU(1));k)\cong \Gamma_k[x]\otimes {\textstyle \bigwedge_k}(b_1,\dots,b_s)
\end{equation}
where $|x|=2$ and $|b_i|=1$. For $I=\{i_1,\dots,i_l\}\subseteq \{1,\dots,s\}$ and an integer $n\geq 0$ define
\[
\xi_{I,n}=\gamma_n(x) b_{i_1}\cdots b_{i_l}\,.
\]
This is a class of degree $|\xi_{I,n}|=|I|+2n$.

\begin{notation}
For a graded $k$-module $V$ we denote by $\Lambda V=\bigoplus_{n\geq 0} (V^{\otimes n})_{\Sigma_n}$ the free graded-commutative $k$-algebra generated by $V$. For a graded set $\mathscr{S}$, we denote by $\Lambda \mathscr{S}$ the free graded-commutative $k$-algebra generated by the $k$-span of $\mathscr{S}$.
\end{notation}

\begin{theorem} \label{thm:quotientoffreealgebra}
Let $k$ be a commutative ring. There is an isomorphism of $k$-algebras
\[
H_\ast(\Kdef(F_s);k) \cong \Lambda(\xi_{I,n} \mid I\subseteq \{1,\dots,s\}, n\geq 0) /\mathscr{I}\, .
\]
Let $a_{ij}$ and $x_j$ be variables of degree $|x_j|=2j$ and $|a_{ij}|=2j-1$. Then the ideal $\mathscr{I}$ coincides with the kernel of the map
\begin{alignat*}{1}
\Lambda(\xi_{I,n} \mid I\subseteq \{1,\dots,s\},n\geq 0) & \to \Lambda(x_j,a_{ij} \mid j\geq 0,\, i\in \{1,\dots,s\}) \\
\xi_{\{i_1,\dots,i_l\},n} & \mapsto  \sum_{(\lambda_0,\dots,\lambda_l)\vdash n} x_{\lambda_0} a_{i_1,1+\lambda_1} \cdots a_{i_l,1+\lambda_l}
\end{alignat*}
where the sum is over all ordered partitions of $n$ into $l+1$ non-negative integers.
\end{theorem}
\begin{proof}

The first part follows directly from Corollary \ref{cor:generation}.

The ideal $\mathscr{I}$ can be understood by considering the group completion map
\[
\gamma\co \Kdef(F_s)\to \Omega B \Kdef(F_s)\, .
\]
Lemma \ref{lem:comultiplication} shows that for every $n\geq 0$ the canonical map
\[
\Map(\vee^s S^1,BU(n))\to \Map(\vee^s S^1,BU(n+1))
\]
is surjective in cohomology, and hence injective in homology. Since
\[
\Omega B \Kdef(F_s) \simeq \Z\times \mathrm{tel}_n \Map(\vee^s S^1,BU(n)) \simeq \Z\times BU \times U^s\,,
\]
it follows that $\gamma$ is injective in homology as well. Therefore, $\mathscr{I}$ can be identified with the kernel of the composite map
\[
\Lambda(\xi_{I,n} \mid I\subseteq \{1,\dots,s\},n\geq 0) \to  H_\ast(\Kdef(F_s);k) \xrightarrow{\gamma_\ast} H_\ast(\Z \times BU \times U^s;k)\, .
\]
Fix primitive generators $a_{i1},a_{i2},\dots$ for the homology of the $i$-th factor of $U$, and let $x_0,x_1,x_2,\dots$ be the usual additive generators for the homology of $\{1\}\times BU(1)\subseteq \Z\times BU$. In particular, $|a_{ij}|=2j-1$ and $|x_n|=2n$. Then
\[
H_\ast(\Z \times BU \times U^s;k)\cong \Lambda(x_j,a_{ij} \mid j\geq 0,\, i\in \{1,\dots,s\})[x_0^{-1}]\,.
\]
If the generators in (\ref{eq:homologydim1}) are chosen accordingly, then the composite map
\begin{equation} \label{eq:idealmap}
f\co \Map(\vee^s S^1,BU(1))\hookrightarrow \Kdef(F_s)\xrightarrow{\gamma} \Omega B \Kdef(F_s)
\end{equation}
sends $b_{i_1}\cdots b_{i_l}$ to $x_0 a_{i_1,1}\cdots a_{i_l,1}$ and $\gamma_n(x)$ to $x_n$, respectively. This determines the induced map on homology fully, because $\Omega B \Kdef(F_s)$ is an $\mathbb{E}_\infty$-ring space (the second $\mathbb{E}_\infty$-structure is induced by the tensor product of representations) and the map $f$ is an $H$-map for the multiplicative $\mathbb{E}_\infty$-structure. Therefore, for $I=\{i_1,\dots,i_l\}\subseteq \{1,\dots,s\}$ one gets
\[
f_\ast(\xi_{I,n})=x_n\circ (x_0 a_{i_1 1}\cdots a_{i_l 1}) = \sum_{(\lambda_0,\dots,\lambda_l)\vdash n} x_{\lambda_0} a_{i_1,1+\lambda_1} \cdots a_{i_l,1+\lambda_l}\, .
\]
The symbol $\circ$ is the Hopf ring product in the Hopf ring $H_\ast(\Kdef(F_s);k)$ (for the notion of \emph{Hopf ring} see \cite{Wilson}). The second equality results from the distributive law in Hopf rings, the Cartan formula $x_n\mapsto \sum_{i+j=n} x_i\otimes x_j$ for the coproduct in $H_\ast(\Z\times BU;k)$, and the structure of $H_\ast(U;k)$ as a module over the Hopf ring $H_\ast(\Z\times BU;k)$ which gives $x_j\circ a_{i1}=a_{i,1+j}$ for all $j$.
\end{proof}

\begin{example} \label{ex:notfree}
When $s=1$, the map of Theorem \ref{thm:quotientoffreealgebra} is injective and thus
\[
H_\ast(\Kdef(\Z);k) \cong k[\xi_{\emptyset,n}\mid n\geq 0]\otimes {\textstyle \bigwedge_k}(\xi_{\{1\},n}\mid n\geq 0)\, .
\]

When $s\geq 2$, one has the following relation in $H_\ast(\Kdef(F_s);k)$
\[
\xi_{\{1\},0} \xi_{\{2\},0} = \xi_{\emptyset,0} \xi_{\{1,2\},0}\,,
\]
since both sides are mapped to $x_0^2 a_{11} a_{21}$ under $f_\ast$ (\ref{eq:idealmap}). This shows that $H_\ast(\Kdef(F_s);k)$ is not a free graded commutative algebra as soon as $s\geq 2$.
\end{example}

Suppose that $k$ is a field of characteristic zero. Then we can view $H_\ast(\Kdef(F_s);k)$ also as a quotient of $H_\ast(\Kdef(\Z^s);k)$. To this end, let $\Sym(V)\in \CAlg_k$ denote the free $\mathbb{E}_\infty$-$k$-algebra generated by $V\in \Mod_k$. Then $k[\Vect_{\C}^\simeq]\simeq \Sym(k[BU(1)])$ (see Lemma \ref{lem:free}), and thus by Proposition \ref{prop:key}
\[
k[\Kdef(\Z^s)] \simeq \Sym(k[BU(1)]) \otimes (S^1)^s\simeq \Sym(k[BU(1)\times (S^1)^s])\, .
\]
By inspection this equivalence is induced by the inclusion of
\[
BU(1)\times (S^1)^s\simeq \Map(\vee^s S^1,BU(1))
\]
into $\Kdef(\Z^s)$, hence it induces an isomorphism
\[
H_\ast(\Kdef(\Z^s);k) \cong \Lambda(\xi_{I,n} \mid I\subseteq \{1,\dots,s\},n\geq 0)\, .
\]
This gives the following viewpoint regarding Theorem \ref{thm:quotientoffreealgebra}.

\begin{corollary} \label{cor:abelianisation}
Let $k$ be a field of characteristic zero. The map $\Kdef(\Z^s)\to \Kdef(F_s)$ induced by abelianisation is surjective in homology with coefficients in $k$.
\end{corollary}

\begin{remark}
It should be noted that on non-equivariant homology the map induced by abelianisation $\Hom(\Z^s,U(n)) \to \Hom(F_s,U(n))\cong U(n)^s$ is neither surjective nor injective in general.
\end{remark}

\subsection{Formality of chains in $\Kdef(F_s)$} \label{sec:proofofformality}

Throughout this section let $k$ be a field of characteristic zero. We will show that the singular $k$-chains in $\Kdef(F_s)$ are formal as an $\mathbb{E}_\infty$-algebra.

\begin{notation}
We denote the $\mathbb{E}_\infty$-space $\Map(\vee^s S^1,\Vect_{\C}^{\simeq})$ by $M_s$.
\end{notation}

Because of the equivalence $\Kdef(F_s)\simeq M_s$ of Proposition \ref{prop:identificationkdef} our goal is to prove an equivalence of $\mathbb{E}_\infty$-$k$-algebras
\[
k[M_s] \simeq H_\ast(M_s;k)\, .
\]
The proof will show that this equivalence is owed to the facts that $k[\Vect_{\C}^\simeq]$ and $k[M_1]$ are free $\mathbb{E}_\infty$-$k$-algebras, and that the cohomology of $\Map(\vee^s S^1,BU(n))$ for each $n$ is a free $H^\ast(BU(n);k)$-module.

For our proof it will be necessary to take into account the $\mathbb{E}_\infty$-$k$-coalgebra structure of $k[M_s]$ as well. Our reference for coalgebra objects in a symmetric monoidal $\infty$-category are Sections 3.1 and 3.3 in \cite{Elliptic}.

\begin{definition} Let $\mathcal{C}$ be a symmetric monoidal $\infty$-category.  An \emph{$\mathbb{E}_\infty$-coalgebra in $\mathcal{C}$} is an $\mathbb{E}_\infty$-algebra in $\mathcal{C}^{\op}$. We let $\cCAlg(\mathcal{C}):=\CAlg(\mathcal{C}^\op)^\op$ denote the $\infty$-category of $\mathbb{E}_\infty$-coalgebras in $\mathcal{C}$.
\end{definition}

\begin{notation}
If $\mathcal{C}=\Mod_k$ with symmetric monoidal product $\otimes_k$, then we shall write $\cCAlg_k$ instead of $\cCAlg(\Mod_k)$. We refer to the objects of $\cCAlg_k$ as \emph{$\mathbb{E}_\infty$-$k$-coalgebras}.
\end{notation}

The $\infty$-category $\cCAlg_k$ inherits a symmetric monoidal structure from $\Mod_k$ such that the forgetful functor $\cCAlg_k \to \Mod_k$ is symmetric monoidal. Since the tensor product in $\Mod_k$ preserves small colimits separately in each variable, it follows from \cite[Corollary 3.1.4]{Elliptic} that the $\infty$-category $\cCAlg_k$ is presentable, and from \cite[3.2.2.5]{HA} (applied to $\Mod_k^\op$) that the forgetful functor $\cCAlg_k \to \Mod_k$ preserves and reflects small colimits. It follows that the tensor product on $\cCAlg_k$ also preserves small colimits in each variable.

We can equip the $\infty$-category of $\N$-graded $\mathbb{E}_\infty$-$k$-coalgebras $\cCAlg_k^\N$ with the Day convolution product (see Section \ref{sec:graded}). We can then take $\mathbb{E}_\infty$-algebras in $\cCAlg_k^\N$; the $\infty$-category $\CAlg(\cCAlg_k^\N)$ is presentable by \cite[3.2.3.5]{HA}.

\begin{definition}
A \emph{charged $\mathbb{E}_\infty$-$k$-bialgebra} is an object of the $\infty$-category $\CAlg(\cCAlg_k^\N)$.
\end{definition}

A charged $\mathbb{E}_\infty$-$k$-bialgebra is what we get naturally by taking the $k$-chains of a charged $\mathbb{E}_\infty$-space: If $\Spc$ is equipped with the Cartesian symmetric monoidal structure (which we assume throughout), the forgetful functor $\cCAlg(\mathcal{S})\to \mathcal{S}$ is an equivalence \cite[2.4.3.10]{HA}; in other words, a space admits the structure of an $\mathbb{E}_\infty$-coalgebra in an essentially unique way (this structure is given by the diagonal). Taking $k$-linear chains induces a symmetric monoidal functor $\Spc\simeq \cCAlg(\Spc)\to \cCAlg_k$ and hence a functor
\[
k[-]\co \CAlg(\Spc^\N) \to \CAlg(\cCAlg_k^\N)\, .
\]
Since $k$ is a field, the K{\"u}nneth isomorphism holds and homotopy $\pi_\ast \co \cCAlg_k\to \cCAlg_k$ is a symmetric monoidal functor. It follows that homology $H_\ast(-;k)\co \Spc\to \Mod_k$ induces a functor from $\CAlg(\Spc^\N)$ to $\CAlg(\cCAlg_k^\N)$ as well.

In particular, $k[M_s]$ and $H_\ast(M_s;k)$ can be viewed as objects of $\CAlg(\cCAlg_k^\N)$.

\begin{definition}
We say that a map $A\to B$ in $\CAlg(\cCAlg_k^\N)$ is \emph{formal} if it fits into a commutative diagram
\[
\xymatrix{
A\ar[r] \ar[d]^-{\sim} & B \ar[d]^-{\sim} \\
\pi_\ast(A) \ar[r] & \pi_\ast(B)\, .
}
\]
\end{definition}

The statement to be proved in this section is:
\begin{proposition} \label{prop:keyformality}
For every $s\geq 0$, the map $k[M_s]\to k[\Vect^\simeq_\C]$ induced by evaluation at the basepoint is formal in $\CAlg(\cCAlg_k^\N)$.
\end{proposition}

Let $p\co \N\to \ast$ be the projection. The forgetful functor $\CAlg(\cCAlg_k)\to \cCAlg_k$ can be factored as
\[
\CAlg(\cCAlg_k) \xrightarrow{p^\ast} \CAlg(\cCAlg_k^\N) \xrightarrow{U} \cCAlg_k^\N \xrightarrow{ev_1} \cCAlg_k\, ,
\]
where $p^\ast$ is induced by the constant diagram functor, $U$ is the forgetful functor, and $ev_1$ is the functor sending an $\N$-graded object $(X_n)_{n\geq 0}$ to $X_1$. As each of these functors preserves small limits \cite[3.2.2.5]{HA} and because the categories are presentable, each functor admits a left-adjoint.

\begin{notation} \label{not:sym}
We denote by $\Sym \co \cCAlg_k \to \CAlg(\cCAlg_k^\N)$ the left-adjoint of $ev_1\circ U$.
\end{notation}

In other words, $\Sym(C)$ is the free charged $\mathbb{E}_\infty$-$k$-bialgebra generated by the graded $\mathbb{E}_\infty$-$k$-coalgebra $(C_n)_{n\geq 0}$ with $C_1=C$ and $C_n=0$ for every $n\neq 1$.

\begin{remark}
Informally, $\Sym(V)\simeq \bigoplus_{n\geq 0} (V^{\otimes n})_{h\Sigma_n}$ with the grading given by $n$. Note that homotopy $\Sigma_n$-orbits are used and thus, coalgebra structures aside, if $V$ is an ordinary $k$-module, then $\Sym(V)$ is not in general equivalent to $\Lambda (V)=\bigoplus_{n\geq 0} (V^{\otimes n})_{\Sigma_n}$, the free graded-commutative $k$-algebra generated by $V$. However, there is always a canonical map of $\mathbb{E}_\infty$-$k$-algebras $\Sym(V)\to \Lambda (V)$ and it is an equivalence when $k$ has characteristic zero, which is what we assume.
\end{remark}

\begin{lemma} \label{lem:free}
The inclusion of spaces $BU(1) \to \Vect_{\C}^{\simeq}$ induces an equivalence
\[
\Sym(k[BU(1)])\simeq k[\Vect_{\C}^{\simeq}]
\]
in $\CAlg(\cCAlg_k^\N)$.
\end{lemma}
\begin{proof}
To check that the induced map $\Sym(k[BU(1)])\to k[\Vect_{\C}^{\simeq}]$ is an equivalence, we can apply the forgetful functor to $\Mod_k^\N$ and check that the resulting map of graded modules is an equivalence in each degree. But in degree $n$ this is the canonical map
\[
(k[BU(1)]^{\otimes n})_{h\Sigma_n} \to k[BU(n)]\,;
\]
since $k$ has characteristic zero, on homotopy this becomes the canonical map
\[
(H_\ast(BU(1);k)^{\otimes n})_{\Sigma_n} \to H_\ast(BU(n);k)\,,
\]
which is an isomorphism.
\end{proof}

\begin{lemma} \label{lem:free2}
The inclusion of spaces $\Map(S^1,BU(1)) \to M_1$ induces an equivalence
\[
\Sym(k[\Map(S^1,BU(1))])\simeq k[M_1]
\]
in $\CAlg(\cCAlg_k^\N)$.
\end{lemma}
\begin{proof}
Similarly as in Lemma \ref{lem:free}, it follows from the case $s=1$ of Example \ref{ex:notfree} that the underlying map of $k$-modules is an equivalence.
\end{proof}

It follows that the map $k[M_1]\to k[\Vect_{\C}^\simeq]$ is equivalent in $\CAlg(\cCAlg_k^\N)$ to the map
\[
\Sym(k[\Map(S^1,BU(1))])\to \Sym(k[BU(1)])
\]
induced by $ev_1\co \Map(S^1,BU(1))\to BU(1)$.

\begin{lemma} \label{lem:formalccalg}
The map $ev_1\co \Map(S^1,BU(1))\to BU(1)$ is formal in $\cCAlg_k$.
\end{lemma}
\begin{proof}
The statement to be proved is that there is a commutative diagram in $\cCAlg_k$
\[
\xymatrix{
k[\Map(S^1,BU(1))]\ar[r] \ar[d]^-{\simeq} &  k[BU(1)] \ar[d]^-{\simeq} \\
H_\ast(\Map(S^1,BU(1));k) \ar[r] & H_\ast(BU(1);k)
}
\]
where both horizontal arrows are induced by $ev_1$. Since each module in the diagram is of finite type over $k$, we may dualise and construct a commutative diagram in $\CAlg_k$ instead. Write $H^\ast(\Map(S^1,BU(1));k)\cong k[x]\otimes \bigwedge(y)$ where $|x|=2$ and $|y|=1$ and such that $ev_1^\ast\co H^\ast(BU(1);k)\to H^\ast(\Map(S^1,BU(1));k)$ becomes the inclusion $k[x]\hookrightarrow k[x]\otimes \bigwedge(y)$. The classes $x$ and $y$ are represented by maps $k[-2]\to \Map_k(k[BU(1)],k)$ and $k[-1]\to  \Map_k(k[\Map(S^1,BU(1))],k)$, which induce a commutative diagram
\[
\begin{tikzcd}
k[x] \arrow[hookrightarrow]{d} & \Sym(k[-2]) \arrow[d] \arrow[r, "\sim"] \arrow[l, "\sim" swap] & \Map_k(k[BU(1)],k) \arrow[d]  \\
k[x]\otimes \bigwedge(y) & \Sym(k[-2]\oplus k[-1]) \arrow[l, "\sim" swap] \arrow[r, "\sim"] & \Map_k(k[\Map(S^1,BU(1))],k)
\end{tikzcd}
\]
as wanted.
\end{proof}

\begin{corollary} \label{cor:basecase}
The map $k[M_1]\to k[\Vect_{\C}^\simeq]$ is formal in $\CAlg(\cCAlg_k^\N)$.
\end{corollary}
\begin{proof}
This follows by applying $\Sym(k[-])$ to the map of Lemma \ref{lem:formalccalg} and using the equivalences of Lemmas \ref{lem:free} and \ref{lem:free2}.
\end{proof}

\begin{notation}
We let $\square$ denote the pullback in the $\infty$-category $\cCAlg_k$.
\end{notation}

We will use the Eilenberg-Moore theorem \cite{EM} in the following form.

\begin{theorem}[Eilenberg-Moore] \label{thm:em}
Let there be given a pullback square in $\mathcal{S}$
\[
\xymatrix{
Z\ar[r] \ar[d] & Y \ar[d] \\
X \ar[r] & B
}
\]
and assume that $B$ is $1$-connected. 
\begin{enumerate}
\item[(1)] The canonical map in $\cCAlg_k$
\[
k[Z] \to k[X]\square_{k[B]} k[Y]
\]
is an equivalence.
\item[(2)] If $H_\ast(Y;k)$ is a free $H_\ast(B;k)$-comodule, then the canonical map in $\cCAlg_k$
\[
H_\ast(Z;k) \to H_\ast(X;k) \square_{H_\ast(B;k)} H_\ast(Y;k)
\]
is an equivalence.
\end{enumerate}
\end{theorem}
\begin{proof}
The first statement is one way of phrasing the classical theorem of Eilenberg and Moore \cite{EM}, who constructed the pullback explicitly as a cobar complex (cf. \cite[Section 5.2.2]{HA}). In the situation to which we will apply the theorem, all spaces have finite type over $k$; in this case one can obtain the statement also from \cite[Corollary 1.1.10]{DAGXIII} by passing to duals.

The second item follows from the collapse of the Eilenberg-Moore spectral sequence
\[
E^2=\mathrm{Cotor}^\ast_{H_\ast(B;k)}(H_\ast(X;k),H_\ast(Y;k)) \Longrightarrow H_\ast(Z;k)\, .
\]
On homotopy, the map in question (displayed in (2)) is the composition of the canonical map into the strict cotensor product
\[
H_\ast(Z;k) \to \mathrm{Cotor}^0_{H_\ast(B;k)}(H_\ast(X;k),H_\ast(Y;k))
\]
followed by the inclusion of the strict into the derived cotensor product. Since $H_\ast(Y;k)$ is a free $H_\ast(B;k)$-comodule, the higher derived cotensor products vanish and the spectral sequence collapses onto the $0$-th column.
\end{proof}

\begin{notation}
We let $\square^\N$ denote the pullback in the $\infty$-category $\CAlg(\cCAlg_k^\N)$.
\end{notation}

\begin{corollary} \label{cor:equivalencechains}
Consider the pullback square in $\CAlg(\Spc^\N)$
\[
\xymatrix{
M_s\ar[r] \ar[d] & M_{1} \ar[d] \\
M_{s-1} \ar[r] & M
}
\]
For every $s\geq 2$ the canonical maps
\begin{enumerate}
\item[(1)] $k[M_s]\to k[M_1]\square^\N_{k[M]} k[M_{s-1}]$
\item[(2)] $H_\ast(M_s;k) \to H_\ast(M_1;k)\square^\N_{H_\ast(M;k)} H_\ast(M_{s-1};k)$
\end{enumerate}
are equivalences in $\CAlg(\cCAlg_k^\N)$.
\end{corollary}
\begin{proof}
This follows from the facts that the forgetful functor $\CAlg(\cCAlg_k^\N)\to \cCAlg_k^\N$ is conservative (see \cite[3.2.2.6]{HA}) and preserves pullbacks (see \cite[3.2.2.5]{HA}), pullbacks in $\cCAlg^\N_k$ are formed degreewise, and for each $n\geq 0$ the canonical maps in $\cCAlg_k$
\begin{enumerate}
\item[(1)] $k[L_sBU(n)] \to k[L_1 BU(n)] \square_{k[BU(n)]} k[L_{s-1}BU(n)]$
\item[(2)] $H_\ast(L_s BU(n);k) \to H_\ast(L_1BU(n);k) \square_{H_\ast(BU(n);k)} H_\ast(L_{s-1}BU(n);k)$
\end{enumerate}
are equivalences by Theorem \ref{thm:em}, where we wrote $L_sBU(n):=\Map(\vee^s S^1,BU(n))$.
\end{proof}

\begin{proof}[Proof of Proposition \ref{prop:keyformality}]
The proof is by induction over $s$. The case $s=0$ just says that $k[\Vect^\simeq_\C]$ is formal, which follows from Lemma \ref{lem:free}. The case $s=1$ is Corollary \ref{cor:basecase}. Now assume that $k[M_{s-1}]\to k[\Vect_\C^\simeq]$ is formal. By Corollary \ref{cor:equivalencechains} there are equivalences in $\CAlg(\cCAlg_k^\N)$
\begin{equation*}
\begin{split}
k[M_s] & \simeq k[M_1]\square^\N_{k[M]} k[M_{s-1}]\\& \simeq H_\ast(M_1;k)\square^\N_{H_\ast(M;k)} H_\ast(M_{s-1};k) \\&\simeq H_\ast(M_s;k)
\end{split}
\end{equation*}
the first and last of which are equivalences over $k[\Vect_\C^\simeq]$ and $H_\ast(\Vect_\C^\simeq;k)$, respectively. The second equivalence covers the equivalence $k[\Vect_\C^\simeq]\simeq H_\ast(\Vect_\C^\simeq)$ by the induction hypothesis. This completes the induction step.
\end{proof}

Proposition \ref{prop:keyformality} now implies Theorem \ref{thm:hh}.


\section{Example: $\Hom(F_s\times \Z^r,U(2))$} \label{sec:example}

We give an application of Theorem \ref{thm:hh}.

\begin{theorem} \label{thm:u2}
For any $r,s\geq 0$, and with coefficients in a field $k$ of characteristic zero, the $U(2)$-equivariant cohomology of $\Hom(F_s\times \Z^r,U(2))$ is a free $H^\ast(BU(2);k)$-module with Poincar{\'e} polynomial
\begin{equation*}
\begin{split}
p(s,r;t)=&(1+t)^{s+r}\left((1+t^3)^s+p_1(r,s;t)+p_2(r,s;t)\right)\, ,
\end{split}
\end{equation*}
where
\begin{alignat*}{1}
p_1(s,r;t) & =\frac{(1+t^2)(1+t)^s((1+t)^r-1)}{2}, \\
p_2(s,r;t) &=\frac{(1-t^2)(1-t)^s((1-t)^r-1)}{2}.
\end{alignat*}
\end{theorem}

When $s=0$ or $s=1$ the Poincar{\'e} polynomial simplifies to the known expression for the Poincar{\'e} polynomial of $\Hom(\Z^r,U(2))$, respectively $\Hom(\Z^{r+1},U(2))$ (see \cite[Example 6.2]{RS19}).

\begin{example} Here are some examples for $s>1$:
\begin{alignat*}{1}
p(2,1;t)&=1+3t+5t^2+10t^3+15t^4+12t^5+7t^6+6t^7+4t^8+t^9\\
p(2,2;t)&=1 + 4 t + 11 t^2 + 28 t^3 + 48 t^4 + 52 t^5 + 44 t^6 + 36 t^7 + 23 t^8 + 8 t^9 + t^{10} \\
p(3,1;t)&=1 + 4 t + 9 t^2 + 20 t^3 + 36 t^4 + 43 t^5 + 40 t^6 + 38 t^7 + 31 t^8 + 16 t^9 + 7 t^{10} \\&\quad + 6 t^{11} + 4 t^{12} + t^{13}\\
p(3,2;t)&=1 + 5 t + 17 t^2 + 50 t^3 + 105 t^4 + 155 t^5 + 183 t^6 + 188 t^7 + 155 t^8 + 91 t^9 + 39 t^{10}\\&\quad  + 18 t^{11} + 11 t^{12} + 5 t^{13} + t^{14}.
\end{alignat*}
\end{example}

\begin{remark}
Note that $F_2\times \Z\cong P_3$ is the pure braid group on three strands. The Poincar{\'e} polynomial of $\Hom(F_2\times \Z,SU(2))$ was computed by Adem and Cohen using geometric arguments in \cite[Theorem 4.12]{AC} (see the erratum \cite{ACerratum}). It can also be deduced from a stable splitting due to Crabb \cite{Crabb}. However, neither approach applies to the $U(2)$-representation variety, because both make use of the particularly simple structure of centralisers in $SU(2)$. 
\end{remark}

To prove the theorem we must compute a small portion of the $r$-fold iterated Hochschild homology of $H_\ast(\Kdef(F_s);k)$. In the course of the proof we will see (Proposition \ref{prop:complexoffree}) that the Hochschild complex which computes the equivariant homology of $\Hom(F_s\times \Z^r,U(n))$ is always a complex of free $H_\ast(BU(n);k)$-comodules. It would be interesting to see if $\Hom(F_s\times \Z^r,U(n))$ is $U(n)$-equivariantly formal for all $n$.

We recall the necessary definitions.

Let $k$ be a field of characteristic zero and $(A,\partial)$ a differential graded commutative $k$-algebra (henceforth called just \emph{dgca}). We shall always assume that our (differential) algebras are concentrated in non-negative degrees.

\begin{definition} \label{def:hochschildcomplex}
The \emph{Hochschild complex} of $(A,\partial)$, denoted $(C(A),d)$, is defined as the total complex $(\mathrm{Tot}^{\oplus}(B(A)),d=d_h+d_v)$ of the bicomplex $(B(A),d_h,d_v)$ given by
\[
B_{p,q}(A)=(A^{\otimes (p+1)})_q
\]
with horizontal differential $d_h\co B_{p,q}(A)\to B_{p-1,q}(A)$ given by
\begin{equation*}
\begin{split}
d_h(a_0\otimes \cdots \otimes a_n) &= \sum_{i=0}^{n-1} (-1)^i a_0\otimes \cdots \otimes a_i a_{i+1} \otimes \cdots \otimes a_n \\&\quad + (-1)^{n+(|a_0|+\cdots + |a_{n-1}|)|a_n|} a_n a_0 \otimes a_1 \otimes \cdots \otimes a_{n-1} 
\end{split}
\end{equation*}
and vertical differential $d_v\co B_{p,q}(A)\to B_{p,q-1}(A)$ given by
\[
d_v(a_0\otimes \cdots \otimes a_n) = (-1)^{p+q} \sum_{i=0}^n (-1)^{|a_0|+\cdots + |a_{i-1}|}a_0\otimes \cdots \otimes \partial a_i \otimes \cdots \otimes a_n\, .
\]
The $n$-th homology of $(C(A),d)$ is denoted $\HH_n(A)$.
\end{definition}
For example, if $A$ has trivial differential $\partial=0$, then $d\co C_1(A)\to C_0(A)$ is a graded commutator $a\otimes b \mapsto ab-(-1)^{|a||b|}ba$, and hence zero. In this case, $\HH_0(A)\cong A$.

There is a multiplicative structure on the Hochschild complex called the shuffle product which makes $C(A)$ a dgca. To define it we let the symmetric group $\Sigma_p$ act (in the graded sense, i.e., with appropriate signs) on $B_{p,\ast}(A)=A\otimes A^{\otimes p}$ by permutation of the last $p$ tensor factors.

\begin{definition} \label{def:shuffleproduct}
The \emph{shuffle product} $B_{p,q}(A)\otimes B_{s,t}(A)\to B_{p+s,q+t}(A)$ is defined by
\begin{equation*}
\begin{split}
&(a_0\otimes \cdots \otimes a_p)\cdot (a_{0}'\otimes \cdots \otimes a_{s}')\\&\quad =\sum_{(p,s)\textnormal{-shuffle } \sigma} (-1)^{(|a_1|+\cdots +|a_p|)|a_{0}'|} \textnormal{sign}(\sigma) \sigma\cdot (a_0a_{0}'\otimes a_1 \otimes \cdots \otimes a_p\otimes a_{1}'\otimes \cdots \otimes a_{s}')\, .
\end{split}
\end{equation*}
\end{definition}

The differential of the Hochschild complex satisfies the Leibnitz rule with respect to the shuffle product, and so the shuffle product makes $\HH_\ast(A)$ into a graded commutative $k$-algebra, cf. \cite[4.2.1]{Loday}.

It will be convenient to use a smaller variant of the Hochschild complex.
\begin{definition} \label{def:normalised}
The \emph{normalised} Hochschild complex is the quotient complex $\bar{C}(A)$ of $C(A)$ with respect to the subcomplex $D(A)\subseteq C(A)$ spanned by those tensors $a_0\otimes a_1\otimes \cdots \otimes a_n$ for which $a_i=1$ for at least one $i\in \{1,\dots,n\}$.
\end{definition}

The subcomplex $D(A)$ is the total complex of a bicomplex, whose horizontal homology is zero (cf. \cite[1.1.15]{Loday}). Hence, $D(A)$ is acyclic and the quotient map $C(A)\to \bar{C}(A)$ is a quasi-isomorphism. Moreover, the shuffle product descends to $\bar{C}(A)$ and therefore $C(A)\to \bar{C}(A)$ is a quasi-isomorphism of dgcas.

In analogy with Definition \ref{def:chargedeinfinity} we introduce the following terminology.
\begin{definition} \label{def:chargedalgebra}
A dgca $(A,\partial)$ together with a direct sum decomposition of complexes $A=\bigoplus_{n\geq 0} A_n$ such that the product takes $A_k \otimes A_l$ to $A_{k+l}$ will be called a \emph{charged} algebra. The elements of $A_n$ are said to have \emph{charge} $n$.
\end{definition}

For example, the homology of $\Vect_{\C}^\simeq$, or $\Kdef(G)$, is a charged algebra (with trivial differential). The charge corresponds to the rank of a vector space or $G$-representation.

\begin{example}
The grading by charge on $H_\ast(\Kdef(F_s);k)$ is easy to describe. Each generator $\xi_{I,n}$ of Theorem \ref{thm:quotientoffreealgebra} has charge one, because it comes from the homology of $\Map(\vee^s S^1,BU(1))$.
\end{example}

If $A$ and $B$ are charged algebras, then so is $A\otimes B$ in a natural way. It follows that if $A$ is a charged algebra, then so are the Hochschild complex $C(A)$, the normalised Hochschild complex $\bar{C}(A)$, and the Hochschild homology $\HH_\ast(A)$.

\begin{notation}
For a charged algebra $A$, we denote by $A_{\geq n}$ the ideal in $A$ generated by all elements of charge $\geq n$. We denote the $n$-truncation $A/A_{\geq n+1}$ by $A_{\leq n}$.
\end{notation}

We observe that if $A$ is a charged algebra, then the quotient map $A\to A_{\leq n}$ induces isomorphisms $C(A)_{\leq n}\cong C(A_{\leq n})_{\leq n}$ and $\HH_\ast(A)_{\leq n}\cong \HH_\ast(A_{\leq n})_{\leq n}$.

\begin{definition} \label{def:iterated}
Let $A$ be a dgca and $r\geq 1$ an integer. The \emph{$r$-fold iterated Hochschild homology} of $A$, denoted $\HH_\ast^{(S^1)^r}(A)$, is defined as the homology of the iterated Hochschild complex $C^{(r)}(A):=C(C(\cdots C(A)\cdots))$ ($r$ times).
\end{definition}

This definition is consistent with that of Section \ref{sec:hh} in that $H_\ast(C^{(r)}(A))\cong \pi_\ast(A\otimes (S^1)^r)$.

It follows inductively that for a charged algebra $A$ and any $n\geq 0$ the projection $A\to A_{\leq n}$ induces an isomorphism $\HH_\ast^{(S^1)^r}(A)_{\leq n}\cong \HH_\ast^{(S^1)^r}(A_{\leq n})_{\leq n}$.

\begin{notation}
Let $V$ be a graded $k$-vector space.
\begin{enumerate}
\item[(i)] We denote by $\Omega V$ the shift of $V$ with $(\Omega V)_{n}=V_{n+1}$. Suppose that $p(X)=a_n X^n + \cdots + a_1 X + a_0$ is a polynomial with non-negative integer coefficients. Then we write $p(\Omega) V$ for $\bigoplus_{i=0}^n (\Omega^{i} V)^{\oplus a_i}$. For example, $(1+\Omega) V$ stands for $V\oplus \Omega V$.
\item[(ii)] We denote by $N\co V\otimes V\to V\otimes V$ the symmetrisation map
\[
N(a\otimes b)=a\otimes b+(-1)^{|a||b|} b\otimes a
\]
and by $[-,-]\co V\otimes V\to V\otimes V$ the graded commutator map
\[
[a,b]=a\otimes b-(-1)^{|a||b|} b\otimes a\, .
\]
\item[(iii)] We denote by $\Alt^2(V)\subseteq V\otimes V$ the subspace of alternating tensors, defined either as the kernel of $N$ or the image of $[-,-]$.
\end{enumerate}
\end{notation}

The next lemma is a formality result which will allow us to compute rather easily the iterated Hochschild homology in charges $\leq 2$.

\begin{lemma} \label{lem:hhcharged}
Let $A=\bigoplus_{n\geq 0} A_{n}$ be a charged $k$-algebra with trivial differential and so that $A_0=k$. Then, there is a quasi-isomorphism of dgcas
\[
C(A)_{\leq 2} \xrightarrow{\sim} \HH_\ast(A)_{\leq 2}\, .
\]
Moreover, as graded $k$-vector spaces
\[
\HH_\ast(A)_i \cong \begin{cases} k & \textnormal{if }i=0, \\
 (1+\Omega)A_1 & \textnormal{if }i=1, \\
 (1+\Omega) A_2 \oplus \Omega(1+\Omega) \Alt^2(A_1) & \textnormal{if }i=2\, .
 \end{cases}
 \]
\end{lemma}
\begin{proof}
The normalised Hochschild complex $\bar{C}(A)_{\leq 2}$ (see Definition \ref{def:normalised}) reads

\[ \xymatrixrowsep{0mm}
\xymatrix{
& & k \\
& & \oplus \\
& A_1 \ar[r]^-{0} & A_1 \\
& \oplus & \oplus \\
A_1\otimes A_1 \ar[r]^-{d} & A_2 \oplus (A_1\otimes A_1) \ar[r]^-{0} & A_2 
}
\]
with $d(a\otimes b)=(-ab, N(a\otimes b))$. The three rows correspond to charge zero, one and two.

We have that $\ker(d)=\Alt^2(A_1)$. Moreover, the projection
\begin{alignat*}{1}
\pi\co A_2\oplus (A_1\otimes A_1)& \to A_2 \otimes \Alt^2(A_1) \\
(x,a\otimes b) & \mapsto (x+\frac{1}{2}ab, [a,b])
\end{alignat*}
induces an isomorphism $\mathrm{coker}(d)\cong A_2\oplus \Alt^2(A_1)$. This proves the computation of $\HH_\ast(A)_i$.

There is a quasi-isomorphism $\bar{C}(A)_{\leq 2}\to \HH_\ast(A)_{\leq 2}$ which is the canonical projection everywhere except for the domain of $d$, where we can take it to be $[-,-]\co A_1\otimes A_1\to \Alt^2(A_1)=\ker(d)$. The only non-trivial products are those between $\bar{C}(A)_0=k$ and $\bar{C}(A)_2$, and those between $\bar{C}(A)_{\leq 1}$ and itself. The quasi-isomorphism is multiplicative on the former, because it is $k$-linear, and it is multiplicative on the latter, because it is the identity in charges $\leq 1$.
\end{proof}

Note that as a $k$-module $\HH_\ast(A)_{\leq 2}$ does not depend on the multiplicative structure of $A$ (though as an algebra $\HH_\ast(A)_{\leq 2}$ does depend on it).

\begin{corollary} \label{cor:hhchargedalgebra}
Let $A$ be a charged $k$-algebra with $A_0=k$. Then, for any $r\geq 0$
\[
\HH^{(S^1)^r}_\ast(A)_{i}  \cong \begin{cases} A_0 & \textnormal{if }i=0, \\  (1+\Omega)^r A_1 & \textnormal{if }i=1, \\ (1+\Omega)^r A_2\oplus \bigoplus_{k=1}^r \Omega(1+\Omega)^k \Alt^2((1+\Omega)^{r-k}A_1) & \textnormal{if } i=2\, . \end{cases} 
\]
\end{corollary}
\begin{proof}
This follows by induction over $r\geq 0$. The base case $r=0$ is clear. Now suppose that the statement is true for all $r< R$ for some integer $R\geq 1$. We have that
\begin{equation*}
\begin{split}
\HH_\ast^{(S^1)^R}(A)_{\leq 2}& = \HH_\ast^{(S^1)^{R-1}}(C(A))_{\leq 2} \\& \cong \HH_\ast^{(S^1)^{R-1}}(C(A)_{\leq 2})_{\leq 2}  \\& \cong \HH_\ast^{(S^1)^{R-1}}(\HH_\ast(A)_{\leq 2})_{\leq 2} \,,
\end{split}
\end{equation*}
where the last isomorphism follows from Lemma \ref{lem:hhcharged}. Now we use the computation from Lemma \ref{lem:hhcharged} and the induction hypothesis to complete the induction step.
\end{proof}

We now turn to the particular case of $A=H_\ast(\Kdef(F_s);k)$.

\begin{proposition} \label{prop:complexoffree}
For every $n\geq 0$, the chain complex $C^{(r)}(A)_n$ (see Definition \ref{def:iterated}) computing $\HH_\ast^{(S^1)^r}(A)_n$ is a chain complex of free $H_\ast(BU(n);k)$-comodules.
\end{proposition}
\begin{proof}
For $i\geq 0$, let $B_i$ denote the coalgebra $H_\ast(BU(i);k)$. For $i,j \geq 0$, we may view $B_i\otimes B_j$ as a $B_{i+j}$-comodule via the natural map $BU(i)\times BU(j) \to BU(i+j)$. It follows from the Leray-Hirsch theorem applied to the homotopy fibre sequence
\[
U(i+j)/U(i)\times U(j) \to BU(i)\times BU(j) \to BU(i+j)
\]
that $B_{i}\otimes B_j$ is a free $B_{i+j}$-comodule. By Lemma \ref{lem:LBUn}, $A_i=H_\ast^{U(i)}(\Hom(F_s,U(i));k)$ is a free $B_i$-comodule. Together it follows that $A_i\otimes A_j$ is a free $B_{i+j}$-comodule and thus, inductively, $A_{i_1}\otimes \cdots \otimes A_{i_l}$ is a free $B_n$-comodule for every partition $i_1+\dots + i_l=n$.

In every degree the chain complex $C^{(r)}(A)_n$ is a direct sum of terms of the form $A_{i_1}\otimes \cdots \otimes A_{i_l}$ with $i_1+\cdots + i_l=n$. The differential is induced by multiplication maps $A_{i}\otimes A_j\to A_{i+j}$ which are maps of $B_{i+j}$-comodules. It follows that this is indeed a chain complex of free $B_n$-comodules.
\end{proof}

We continue to write $B_i=H_\ast(BU(i);k)$. Let $W=\bigoplus_{n\geq 0} W_n$ be a graded $k$-module of finite type.

\begin{notation}
We denote by $p_W(t)=\sum_{n\geq 0} \dim(W_n) t^n$ the Poincar{\'e} series of $W$. If $V=W\otimes B_i$ is a free $B_i$-comodule, we denote by $p_V(t):=p_W(t)$ the Poincar{\'e} series of $V$ as a $B_i$-comodule.
\end{notation}

Since $2$ is invertible in $k$, there is an isomorphism
\begin{equation}\label{eq:splittingsigma2}
W\otimes W\cong (W\otimes W)^{\Sigma_2} \oplus \Alt^2(W)\,.
\end{equation}

\begin{lemma}\label{lem:poincarepolynomial}
The following identities hold:
\begin{enumerate}
\item[(1)] $p_{(W\otimes W)^{\Sigma_2}}(t)=\frac{1}{2}\left( p_W(t)^2+p_W(-t^2)\right)$
\item[(2)] $p_{\Alt^2(W)}(t)=\frac{1}{2}\left(p_W(t)^2-p_W(-t^2)\right)$.
\end{enumerate}
\end{lemma}
\begin{proof}
Let $\mathcal{B}$ be a graded basis for $W$ and let $\leq$ be a total order on $\mathcal{B}$. The first identity follows from the fact that $(W\otimes W)^{\Sigma_2}$ has basis $v\otimes w+w\otimes v$ for all $v < w\in \mathcal{B}$ and $v\otimes v$ for $v\in \mathcal{B}$ with $v$ of even degree. The second identity follows from the first and (\ref{eq:splittingsigma2}).
\end{proof}

Consider the free $B_1$-comodule $V=W\otimes B_1$ and view $V\otimes V$ as a free $B_2$-comodule. Since $N\co V\otimes V\to V\otimes V$ is a map of $B_2$-comodules, its kernel $\Alt^2(V)\subseteq V\otimes V$ is a $B_2$-subcomodule.

\begin{lemma} \label{lem:alt2cofree}
The subcomodule $\Alt^2(V)\subseteq V\otimes V$ is free with Poincar{\'e} polynomial
\[
p_{\Alt^2(V)}(t)=\frac{1}{2}((1+t^2)p_W(t)^2-(1-t^2)p_W(-t^2))\, .
\]
\end{lemma}
\begin{proof}
As in (\ref{eq:splittingsigma2}) we have that
\[
\Alt^2(V) \cong (\Alt^2(W)\otimes (B_1\otimes B_1)^{\Sigma_2}) \oplus ((W\otimes W)^{\Sigma_2}\otimes \Alt^2(B_1))
\]
as $B_2$-comodules. The comodule $(B_1\otimes B_1)^{\Sigma_2}$ is isomorphic to $(B_1\otimes B_1)_{\Sigma_2}\cong B_2$. The comodule $\Alt^2(B_1)$ is free on a single generator in degree two. One way to see this is to note that the graded dual $\Alt^2(B_1)^\vee$ viewed as a module over $B_2^\vee$ is the space of alternating polynomials in $k[x,y]$ (with $|x|=|y|=2$) viewed as a module over the ring of symmetric polynomials $k[x,y]^{\Sigma_2}$. This is a free module generated by $x-y$.

It follows that
\[
p_{\Alt^2(V)}(t)=p_{\Alt^2(W)}(t)+t^2 p_{(W\otimes W)^{\Sigma_2}}(t)\, .
\]
The claim follows now from Lemma \ref{lem:poincarepolynomial}.
\end{proof}

\begin{proof}[Proof of Theorem \ref{thm:u2}]
Theorem \ref{thm:hh} shows that
\[
H_\ast^{U(2)}(\Hom(F_s\times \Z^r,U(2));k)\cong \HH^{(S^1)^r}(A)_2
\]
for $A=H_\ast(\Kdef(F_s);k)$. By Corollary \ref{cor:hhchargedalgebra} and Lemma \ref{lem:alt2cofree}, this is a free $H_\ast(BU(2);k)$-comodule. Its Poincar{\'e} polynomial (as a $H_\ast(BU(2);k)$-comodule) is computed as follows. Set $W=H_\ast(\Hom(F_s,U(1));k)$, so that $A_1=W\otimes B_1$. By Corollary \ref{cor:hhchargedalgebra}, 
\[
p(s,r;t) =(1+t)^r p_{A_2}(t)+\sum_{k=1}^r t(1+t)^k p_{\Alt^2((1+\Omega)^{r-k}W\otimes B_1)}(t) \, .
\]
Now $p_{A_2}(t)=(1+t)^s (1+t^3)^s$ and $p_{(1+\Omega)^{r-k}W}(t)=(1+t)^{r-k}(1+t)^s$. Finally, use Lemma \ref{lem:alt2cofree} and the identities $\sum_{k=1}^r t(1+t)^{2r-k}=((1+t)^r-1)(1+t)^r$ and $\sum_{k=1}^r t(1-t)^{r-k}=1-(1-t)^r$ to determine the sum.
\end{proof}

We finish the section with the Poincar{\'e} polynomial of $\Rep(F_2\times \Z^r,U(2))$.

\begin{theorem} \label{thm:poincarerep}
For every $r\geq 0$, the Poincar{\'e} polynomial of $\Rep(F_2\times \Z^r,U(2))$ is
\[
p(r;t)=(1+t)^{r+2}\left(1+\frac{(1+t)^2((1+t)^r-1)+(1-t)^2((1-t)^r-1)}{2}\right)\, .
\]
\end{theorem}
\begin{proof}
Let $A=H_\ast(\Kdef(F_2);k)$ and $A'=H_\ast(\Rdef(F_2);k)$, both viewed as charged $k$-algebras. We first claim that the canonical map $A\to A'$ is surjective in charges $\leq 2$. This is clear in charge zero where both are $k$; in charge one the map is induced by the projection $\Hom(F_2,U(1))_{hU(1)}\to \Rep(F_2,U(1))$, thus surjective as well.

In charge two we use the fact that the covering $SU(2)\times S^1\to U(2)$ induces a covering
\[
\Rep(F_2,SU(2))\times (S^1)^2 \to \Rep(F_2,U(2))
\]
by \cite[Theorem A]{Lawton}. By \cite[Theorem 6.5]{Florentino}, $\Rep(F_2,SU(2))\cong \Delta^3$, hence contractible. It follows that the determinant map $\mathrm{det}\co U(2)\to S^1$ induces a homotopy equivalence $\Rep(F_2,U(2))\simeq (S^1)^2$. The map $A_2\to A'_2$ is induced by the horizontal map in
\[
\xymatrix{
\Hom(F_2,U(2))_{hU(2)} \ar[r] \ar[dr] & \Rep(F_2,U(2)) \ar[d]^-{\sim} \\
& (S^1)^2
}
\]
But the diagonal arrow has a section and therefore $A_2\to A'_2$ is surjective as well.

If $V\to W$ is a surjective map of graded vector spaces, then so is $\Alt^2(V)\to \Alt^2(W)$. It follows from Corollary \ref{cor:hhchargedalgebra} that the map
\[
\HH_\ast^{(S^1)^r}(A)_2\to \HH_\ast^{(S^1)^r}(A')_2
\]
is surjective. By naturality (see Remark \ref{rem:naturality}), it follows that the spectral sequence of Proposition \ref{prop:ss} 
\[
E^2=\HH^{(S^1)^r}_\ast(A')_{2} \Longrightarrow H_\ast(\Rep(F_2\times \Z^r,U(2));k)
\]
degenerates at the $E^2$-page. Finally, the Poincar{\'e} polynomial of $\HH^{(S^1)^r}_\ast(A')_{2}$ can be computed from Corollary \ref{cor:hhchargedalgebra}; we have $p_{A'_1}(t)=(1+t)^2$, $p_{A'_2}(t)=(1+t)^2$; then use Lemma \ref{lem:poincarepolynomial}.
\end{proof}


\end{document}